\documentclass[12pt]{article}
\usepackage{graphicx}
\usepackage{color}
\usepackage{amsmath,amsthm}
\usepackage{amssymb}
\usepackage{ifpdf}

\theoremstyle{plain}
\newtheorem{thm}{Theorem}[section]

\newtheorem{prop}[thm]{Proposition}

\theoremstyle{definition}

\theoremstyle{remark}

\begin{document}

\title{Equilibria of biological aggregations with nonlocal repulsive-attractive interactions}
\author{R. C. Fetecau $^{\ast}$
\and Y. Huang \thanks{Department of Mathematics, Simon Fraser University, 8888
University Dr., Burnaby, BC V5A 1S6, Canada, email: $\{$van,yha82$\}$@sfu.ca}
 }
\maketitle

\begin{center}
\textbf{Abstract}
\end{center}

\begin{quote}
{\small We consider the aggregation equation $\rho_{t}-\nabla\cdot\left(
\rho\nabla K\ast\rho\right) =0$ in $\mathbb{R}^{n}$, where the interaction
potential $K$ incorporates short-range Newtonian repulsion and long-range power-law attraction. 
We study the global well-posedness of solutions and investigate analytically and numerically 
the equilibrium solutions. We show that there exist unique equilibria supported on a ball of $\mathbb{R}^n$.  
By using the  method of moving planes we prove that such equilibria are radially symmetric and monotone in 
the radial coordinate. We perform asymptotic studies  for the limiting cases when the exponent of the 
power-law attraction approaches infinity and a Newtonian singularity, respectively. Numerical simulations 
suggest that equilibria studied here are global attractors for the dynamics of the aggregation model.}
\end{quote}

\smallskip



\section{Introduction}
\label{sect:intro}

The multidimensional integro-differential equation,
\begin{equation}
\label{eq:aggre}
 \rho_t - \nabla\cdot(\rho\nabla K*\rho) =0,
\end{equation}
has attracted a great amount of interest in recent years. The equation appears in various
contexts related to mathematical models for biological aggregations, where  $\rho$ represents the density 
of the aggregation and $K$ is the social interaction potential. The asterisk $\ast$ denotes convolution. 
We refer to~\cite{M&K, TBL} for an extensive background and literature review on mathematical models of 
social aggregations and in particular, for a thorough discussion on the relevance of equation~\eqref{eq:aggre} 
for modelling swarming behaviours. The equation also arises in a number of
other applications such as granular media~\cite{Toscani2000, CaMcVi2006}, self-assembly of nanoparticles~\cite{HoPu2005,
HoPu2006}, Ginzburg-Landau vortices~\cite{E1994, DuZhang03, MaZh2005} and molecular dynamics simulations 
of matter~\cite{Haile1992}.  In this work however we are primarily interested in biological applications, where equation
~\eqref{eq:aggre} is used to model social aggregations such as insect swarms, fish
schools, bacterial colonies, etc~\cite{M&K}. 

Regarded as a model for biological aggregations, equation~\eqref{eq:aggre} incorporates inter-individual social 
interactions such as long-range attraction and short-range repulsion, through the aggregation potential $K$.  
The properties of the potential (symmetry, regularity, monotonicity, etc) are essential in studying issues 
such as the well-posedness~\cite{BodnarVelasquez2, BertozziLaurent, BertozziLaurentRosado} or the long-time 
behaviour~\cite{Burger:DiFrancesco, LeToBe2009}  of solutions to model equation~\eqref{eq:aggre}.
In particular, a large component of the research on this model dealt with attractive potentials $K$ which lead 
to solutions that blow-up (in finite or infinite time) by mass concentration, into one or several Dirac
distributions~\cite{FeRa10,BertozziCarilloLaurent, HuBe2010}. 

It is essential however for an aggregation model to be able to capture solutions with
biologically relevant features. As pointed out by Mogilner and Keshet in their
seminal work~\cite{M&K} on the class of models discussed here, such desired
characteristics include: finite densities, sharp boundaries, relatively
constant internal population and long lifetimes. The difficulty in finding such solutions to model \eqref{eq:aggre}
 has been indicated as a ``challenge" in previous literature  \cite{Topaz:Bertozzi,LiRodrigo2009}, and in fact there is only a handful of
works that address this issue. Topaz and collaborators~\cite{LeToBe2009, BeTo2011} derived explicit
swarm equilibria that arise in the one-dimensional model with Morse-type
potentials (in the form of decaying exponentials), but their explicit calculations do not extend to higher dimensions. Other works illustrate
asymptotic vortex states in $2D$~\cite{Topaz:Bertozzi} and clumps
(aggregations with compact support) in a nonlocal model that includes
density-dependent diffusion~\cite{TBL}.

A recent publication of the authors~\cite{FeHuKo11} considered an interaction potential 
$K$ for which equilibria of the aggregation model~\eqref{eq:aggre} have the desired characteristics indicated above.
More specifically, the kernel investigated in~\cite{FeHuKo11} has a repulsion component in the form of the Newtonian 
potential and attraction given by a power law\footnote{See Section \ref{sect:discussion} for a discussion on how the potential can be modified to avoid the biologically unrealistic growth of attraction with distance when $q>0$.}:
\begin{equation}
\label{eq:kernel}
K(x) = \phi(x) + \frac{1}{q} |x|^q.
\end{equation}
Here, $\phi(x)$ is the free-space Green's function of the negative Laplace operator  $-\Delta$:
\begin{equation}
\phi(x)=
\begin{cases}
-\frac{1}{2} |x|, & n=1\\
-\frac{1}{2\pi}\ln|x|, & n=2\\
\frac{1}{n(n-2)\omega_{n}}\frac{1}{|x|^{n-2}},\quad & n\geq3,
\end{cases}
\label{eqn:phi}
\end{equation}
and $q$ is a real exponent, $q \geq 2$. In~\eqref{eqn:phi}, $n$ is the number of space dimensions and $\omega_n$ denotes 
the volume of the unit ball in $\mathbb{R}^n$. 

We summarize briefly some of the results from~\cite{FeHuKo11} that are relevant to the present article. For $q=2$, 
the equilibrium density of~\eqref{eq:aggre}-\eqref{eq:kernel} is uniform inside a ball of $\mathbb{R}^n$ and zero 
outside it. In this case, the method of characteristics was used to solve explicitly the dynamics corresponding to 
radially symmetric initial conditions in any dimension. This showed the global stability within the class
of radially symmetric solutions of the constant steady state. The explicit calculations did not extend to general 
exponent $q>2$,  but the existence of a unique radially symmetric equilibrium of compact support was shown, after 
casting the equilibrium problem as an eigenvalue problem for an integral operator and applying the Krein-Rutman 
theorem. Some explicit calculations of the equilibria could be performed  however for the special subcase when $q$ 
is even. In addition to these studies on equilibria, the global well-posedness of solutions 
to~\eqref{eq:aggre}-\eqref{eq:kernel} (with $q \geq 2$) was shown by borrowing  techniques used in the 
analysis of incompressible Euler equations \cite{MajdaBertozzi}. 

The main purpose of the work from~\cite{FeHuKo11} was to design attractive-repulsive potentials that yield 
equilibrium states of finite densities and compact support. In this respect, the attraction component $\frac{1}
{q}|x|^{q}$ of the potential was specifically designed to counter-balance the singular Newtonian repulsion. 
Newtonian (attractive) potentials for model~\eqref{eq:aggre} were considered in 
~\cite{DuZhang03, MaZh2005} in the context of vortex motions in two-dimensional
superfluids. The main concern of these works was the well-posedness of solutions, in particular concentration 
and singularity formation in measure-valued solutions.  Very recently, Newtonian potentials were also considered 
in aggregation models~\cite{BertozziGarnettLaurent,BertozziLaurentLeger}. In~\cite{BertozziLaurentLeger}, the authors 
study patch solutions and they consider separately the case of an attractive Newtonian potential (with finite time 
concentration) and of a repulsive Newtonian potential (with spreading to a circular/ spherical aggregation patch). 

The purpose of the present research is (i) to extend the interaction potential \eqref{eq:kernel}-\eqref{eqn:phi} studied in~\cite{FeHuKo11} to 
allow for more general attractive forces\footnote{In the special case $q=0$ we take $K(x) = \phi(x) + \ln |x|$.} ($q>2-n$) and (ii) to investigate analytically and numerically the 
properties of the equilibria to the aggregation model~\eqref{eq:aggre}-\eqref{eq:kernel} for $q >2-n$. Remarkably, 
the intricate balance between the power-law attraction and the singular repulsion provides the model with a very 
interesting and at the same time biologically relevant set of steady states. For all values of $q \in (2-n,\infty)$, 
the aggregation model has a unique steady state supported in a ball. This steady state is radial and monotone in the 
radial coordinate. More specifically, the equilibria are decreasing about the origin for $2-n<q<2$ and increasing 
for $q>2$, while $q=2$ corresponds to a constant equilibrium density. Figure~\ref{fig:variousq} shows the equilibrium 
solutions in three dimensions for various values of $q$; all shown equilibria have mass 1. The limits $q \to \infty$ 
and $q \searrow 2-n$, that is, when attraction becomes infinitely strong (at large distances) or as singular as the (Newtonian) repulsion,  
are particularly interesting. As $q\to \infty$, the radii of the equilibria approach a constant, but the qualitative 
features change dramatically, as mass aggregates toward the edge of the swarm, leaving an increasingly void region 
in the centre --- this effect can be observed in Figure~\ref{fig:variousq} ($q=20,40,80$). As $q \searrow 2-n$, the 
radii of equilibria approach $0$ and mass concentrates at the origin -- see Figure~\ref{fig:variousq} ($q=1.5,1,0.5$). 
Numerical  simulations suggest that all these equilibria are global attractors for the dynamics 
of~\eqref{eq:aggre}-\eqref{eq:kernel}, which motivates and gives strong grounds to the studies of the present work.
 
 \begin{figure}[htb]
 \label{fig:variousq}
 \begin{center}
  \includegraphics[totalheight=0.28\textheight]{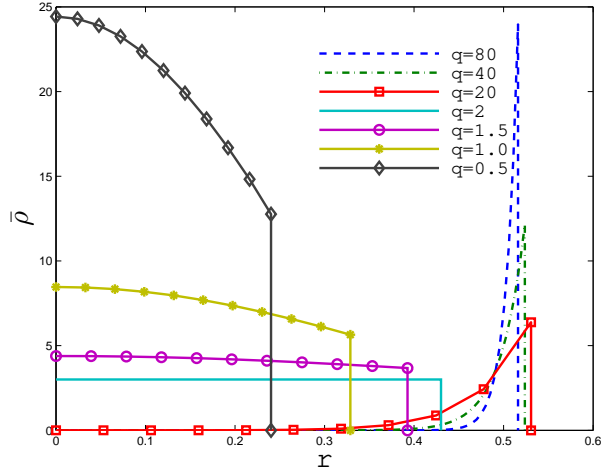}
 \end{center}
\caption{Radially symmetric equilibria of~\eqref{eq:aggre}-\eqref{eq:kernel} in three dimensions, for 
various values of $q$.The equilibria are monotone in the radial coordinate: decreasing about the origin 
for $2-n<q<2$, increasing for $q>2$, and constant for $q=2$. As $q\to \infty$, the radii of the equilibria 
approach a constant, and mass aggregates toward the edge of the swarm, leaving an increasingly void region in 
the centre. As $q \searrow 2-n$, the radii of equilibria approach $0$ and mass concentrates at the origin. 
Numerics suggests that all these equilibria are global attractors for the dynamics of~\eqref{eq:aggre}-\eqref{eq:kernel}.}
\end{figure}

There are a few  very recent studies of equilibria of~\eqref{eq:aggre} with attractive-repulsive potentials 
in power-law form that closely relate to our work. In~\cite{Balague_etal}, the authors study the stability 
of spherical shell equilibria for potentials in the form $K(x) = \frac{1}{q} |x|^q  - \frac{1}{p} |x|^p$, 
where $2-n<p<q$ (short range repulsion and long range attraction). While the attraction component, $|x|^q/q$, 
is the same as in~\eqref{eq:kernel}, the singularity of the repulsion term is ``better" than Newtonian ($p>2-n$). 
Shell steady states are shown to exist and be locally stable under certain conditions on the exponents $p$ and $q$. 
The methods from~\cite{Balague_etal} do not apply to Newtonian singularities. Other recent works that involve 
model~\eqref{eq:aggre} with power-law potentials focus on pattern formation and linear stability analysis of spherical 
shells~\cite{KoSuUmBe2011, Brecht_etal2011}.

 The results of the paper are as follows. Well-posedness of solutions to the aggregation 
model~\eqref{eq:aggre}-\eqref{eq:kernel} (with $q >2-n$) is studied in Section~\ref{sect:well-posedness} by using 
analogies with the incompressible fluid flow equations~\cite{MajdaBertozzi, FeHuKo11, BertozziLaurentLeger}. For 
all values of $q \in (2-n,\infty)$, we show in Section~\ref{subsect:exist} that there exist unique equilibria 
supported on a ball of $\mathbb{R}^n$. In Section~\ref{subsect:monotonicity} we employ the {\em method of moving 
planes}~\cite{GidasNiNirenberg} to prove that such equilibria are radially symmetric and monotone in the radial 
coordinate. In Section~\ref{sect:asy-num} we performed careful asymptotic and numerical investigations of the 
equilibria. Our studies address two issues. The first  one is the behaviour of equilibria as $q \to \infty$ 
and $q\to 2-n$. As expected, the two limiting cases give very different asymptotic behaviours.  The second issue 
addressed in Section~\ref{sect:asy-num} is  the stability of the equilibrium solutions.  The results regarding stability  are preliminary and entirely based on numerical observations. Based on all numerical 
experiments we performed, we conjecture that the equilibria studied in this paper are global attractors for 
solutions to~\eqref{eq:aggre}-\eqref{eq:kernel}.

\section{A priori bounds and well-posedness of solutions}
\label{sect:well-posedness}

We start by pointing out that the aggregation model~\eqref{eq:aggre}-\eqref{eq:kernel} has two 
important conservation properties. Denote the initial density by $\rho_{0}$:
\[
\rho(x,0)=\rho_{0}(x),\qquad x\in\mathbb{R}^{n}.
\]
The aggregation model~\eqref{eq:aggre}-\eqref{eq:kernel} satisfies:

(i) Conservation of mass:
\begin{equation}
\label{eqn:c-mass}
\int\rho(x,t) dx = M, \quad\text{ for all } t \geq0,
\end{equation}
where the constant $M$ denotes the initial mass $M = \int\rho_{0}(x) dx$.

\smallskip(ii) Conservation of centre of mass:
\begin{equation}
\int x\rho(x,t)dx=0,\quad\text{ for all }t\geq0, 
\label{eqn:c-cmass}%
\end{equation}
where we assume, without loss of generality, that the centre of mass of the
initial density is at the origin: $\int x\rho_{0}(x)dx=0$. 

Both properties follow directly from~\eqref{eq:aggre}. Property (ii) uses the radial symmetry of 
the potential. The two conservation properties will be used frequently in this article.

By introducing the notation:
\begin{align*}
f(x)&=-\nabla K(x), \\
\end{align*}
we write the aggregation model as %
\begin{subequations}
\label{ag2}%
\begin{gather}
\rho_{t}+\nabla\cdot(\rho v)=0,
\label{eqn:model2}\\
v=f \ast\rho, 
\label{eqn:v2} \\
f(x)=\left(  \frac{1}{n\omega_{n}}\frac{1}{|x|^{n-1}}-|x|^{q-1}\right)  \frac{x}{|x|}. 
\label{eqn:f-exp}
\end{gather}
\end{subequations}

This work makes extensive use of the Lagrangian formulation of the aggregation model~\eqref{ag2}, where 
dynamics is tracked along the characteristic curves, defined by:
\begin{equation}
\frac{d}{dt}X(\alpha,t)=v(X(\alpha,t),t),\qquad X(\alpha,0)=\alpha,
\label{eqn:char}%
\end{equation}
with velocity $v$ defined by~\eqref{eqn:v2} and \eqref{eqn:f-exp}.


\subsection{A priori bounds on density}
\label{subsect:bounds}
Expand $\nabla \cdot (\rho v) = v \cdot \nabla v + \rho \nabla \cdot v$ and write the evolution of 
the density $\rho(X(\alpha,t),t)$ along characteristics:
\begin{equation}
\label{eqn:ch-form}
\frac{d \rho}{dt} (X(\alpha,t),t)= - \rho \nabla \cdot v (X(\alpha,t),t). 
\end{equation}

Calculation of $\nabla \cdot v$ from \eqref{eqn:v2} and \eqref{eqn:f-exp} yields:
\begin{equation}
\label{eqn:div-vqneq2}
\nabla\cdot v = \rho- (n+q-2)\int_{\mathbb{R}^{n}}|x-y|^{q-2}\rho(y)dy.
\end{equation}

The continuity equation~\eqref{eqn:model2} expresses the fact that
\begin{equation}
\label{eqn:rhoJ}
\rho(X(\alpha,t),t) J(\alpha,t) = \rho_{0}(\alpha),
\end{equation}
where
\[
J(\alpha,t) = \operatorname{det} \nabla_{\alpha}X
\]
is the Jacobian of the particle map $\alpha\to X(\alpha,t)$.

Define the maximum of the density 
\[
 \rho_{\text{max}}(t) = \sup_x \rho(x,t).
\]
We show that provided $\rho_{\text{max}}$ is bounded initially, it remains bounded above, uniformly in time. 
This was shown in~\cite{FeHuKo11} for $q \geq 2$ and here we extend the results to include $2-n<q<2$.

Let $q \in (2-n, 2)$. From the characteristic equation~\eqref{eqn:ch-form} for $\rho$ and the 
expression~\eqref{eqn:div-vqneq2} for $\nabla \cdot v$, we have, along particle trajectories:
\begin{align}
\frac{d\rho}{dt} &= 
 -\rho^2 +(n+q-2)\rho\int_{\mathbb{R}^n}|x-y|^{q-2}\rho(y)dy \label{eqn:rho-evol} \\
&=-\rho^2
+(n+q-2)\rho\left [\int_{|x-y| < r_*}|x-y|^{q-2}\rho(y)dy 
+\int_{|x-y| > r_*}|x-y|^{q-2}\rho(y)dy \right] \nonumber \\
&\leq -\rho^2
+(n+q-2)\rho\left[ \rho_{\text{max}} \int_{|x-y| < r_*} |x-y|^{q-2}dy
+r_*^{q-2} \int_{|x-y| > r_*} \rho(y)dy\right], \label{est:drho_dt}
\end{align}
where we used $q<2$ and $r_*>0$ will be chosen conveniently later. Use the following estimates on 
the integrals in the right-hand-side of~\eqref{est:drho_dt}:
\[
\int_{|x-y| <  r_*}|x-y|^{q-2} dy = \frac{n \omega_n}{n+q-2} \, r_*^{n+q-2}, \qquad \int_{|x-y| > r^*} \rho(y)dy \leq M,
\]
and choose $r_*$ to be:
\begin{equation}
\label{eqn:r*}
r_* = (\rho_{\text{max}})^{-1/n}.
\end{equation}
We find
\begin{equation*}
\frac{d\rho}{dt} \leq -\rho^2 + C\rho\rho_{\text{max}}^{(2-q)/n},
\end{equation*}
with 
\[
C =  n \omega_n+ M (n + q -2),
\]
resulting in the following inequality
\begin{equation}
\label{drho-max_dt}
\frac{d \rho_{\text{max}}}{dt} \leq C\rho_{\text{max}}^{(n+2-q)/n}-\rho_{\text{max}}^2.
\end{equation}
As $q>2-n$, the damping dominates the growth term in the right-hand-side of~\eqref{drho-max_dt}: $(n+2-q)/n <2$. 
The right-hand-side of~\eqref{drho-max_dt} becomes negative when $\rho_{\text{max}} > C^{(n+q-2)/n}$ and hence, 
regardless of the size of the support, the maximum density is bounded uniformly in time, provided it is initially bounded. 
\smallskip

{\bf Remark.} For $q \geq 2$, it was shown in~\cite{FeHuKo11} that the density has compact support uniformly in 
time, provided the initial density $\rho_0$ has compact support. This property was used to conclude global wellposedness 
of solutions. In the present study, where $2-n<q<2$, we could not show the compact support of solutions when $2-n<q\leq1$, 
but we managed to circumvent this by using the uniform $L^1$-bound of $\rho$.

We present briefly the argument that shows uniform compact support for $1<q<2$.  The density is
transported along characteristics (see equation \eqref{eqn:rhoJ}), so it is
enough to show that the trajectories $X(\alpha,t)$ that carry non-zero
densities remain within some compact set. Calculate using~\eqref{eqn:v2} and~\eqref{eqn:f-exp}:
\begin{equation}
\label{eq:sysmb1}x \cdot v(x,t) = \int_{\mathbb{R}^{n}} \frac{x \cdot
(x-y)}{n\omega_{n} |x-y|^{n}} \rho(y,t)dy -\int_{\mathbb{R}^{n}}x
\cdot(x-y)|x-y|^{q-2}\rho(y,t) dy.
\end{equation}
Define the maximum radius of support $R(t)$ as 
\[
R(t) = \max_{\alpha: \rho_0(\alpha) \ne 0} |X(\alpha,t)|,
\]
and evaluate~\eqref{eq:sysmb1} at $x=X(\alpha,t)$ on the boundary of the
support, i.e., $|X(\alpha,t)| = R(t) \geq|X(\beta,t)|$ for any $\beta$ such
that $\rho_{0}(\beta)>0$. The left-hand-side of~\eqref{eq:sysmb1} becomes
\[
X(\alpha,t) \cdot \frac{d}{dt} X(\alpha,t) = R \, \frac{dR}{dt}.
\]
We estimate the first term in the right-hand-side of~\eqref{eq:sysmb1} as follows:
\begin{align*}
\int_{\mathbb{R}^{n}} \frac{x\cdot(x-y)}{n\omega_{n}
|x-y|^{n}} \rho(y,t)dy & \leq \frac{R}{n \omega_n} \int_{|x-y|<1} \frac{1}{|x-y|^{n-1}} \rho(y) dy + 
\frac{R}{n \omega_n}\int_{|x-y|>1} \frac{1}{|x-y|^{n-1}} \rho(y) dy \nonumber \\ 
& \leq  \rho _{\text{max}} R + \frac{M}{n \omega_n}R.
\end{align*}
For the second term, use $|x-y| \leq 2R$ and $x \cdot (x-y) \geq 0$, for $|x|=R$ on the boundary of the support 
and $y$ in the support of $\rho$, to find
\begin{align*}
-\int_{\mathbb{R}^{n}}x \cdot(x-y)|x-y|^{q-2}\rho(y,t) dy & \leq -( 2R)^{q-2} \int_{\mathbb{R}^{n}}x \cdot(x-y) \rho(y) dy \\
& = - 2^{q-2} M R^q,
\end{align*}
where we used conservation of mass~\eqref{eqn:c-mass} and centre of mass~\eqref{eqn:c-cmass} to go from the second to 
the last line. Using the estimates in~\eqref{eq:sysmb1}, derive
\[
\frac{d R }{dt} \leq  \rho _{\text{max}} + \frac{M}{n \omega_n} -  2^{q-2} M R^{q-1}.
\]
Hence, the trajectories that carry non-zero densities will remain inside the disk of radius $R_{\text{max}}$, where 
\[
R_{\text{max}} = \left( \frac{\rho_{\text{max}} + M/(n \omega_n)}{2^{q-2}M} \right)^{\frac{1}{q-1}}.
\]


\subsection{Existence and uniqueness of solutions}
\label{subsect:ex-un}

To study well-posedness of solutions we use a Lagrangian approach and rewrite  the aggregation equation~\eqref{eq:aggre} 
in terms of particle trajectories. Then we regard the model as an ODE on a certain Banach space and infer local existence 
and uniqueness from Picard theorem. The setup of the ODE framework is the same as that used in~\cite{FeHuKo11} to study 
the case $q\geq 2$ and is inspired from the study of well-posedness of solutions to the incompressible Euler equation in 
Lagrangian formulation~\cite{MajdaBertozzi}. Extension to global existence is achieved through an argument similar to 
the well-known Beale-Kato-Majda blow-up criterion for incompressible flows~\cite{BealeKatoMajda}.

Make the change of variable $y=X(\beta,t)$ in the expression~\eqref{eqn:v2} for $v$, with $f$ given by~\eqref{eqn:f-exp}, 
and use~\eqref{eqn:rhoJ} to write the characteristic equation~\eqref{eqn:char} as
\begin{subequations}
\label{eqn:charODE}%
\begin{align}
\frac{d}{dt} X(\alpha,t)  &  = \mathcal{F}(X(\alpha,t))\\
X(\alpha,0)  &  = \alpha,
\end{align}
\end{subequations}
where the map $\mathcal{F}(X)$ is defined by
\begin{equation}
\label{eqn:F-qgen}
\mathcal{F}(X(\alpha,t)) = \int_{\mathbb{R}^{n}} \left(  \frac{1}{n
\omega_{n}} \frac{X(\alpha,t) - X(\beta,t)} {|X(\alpha,t) - X(\beta,t)|^{n}} -
|X(\alpha,t) - X(\beta,t)|^{q-2} (X(\alpha,t) - X(\beta,t)) \right)  \rho
_{0}(\beta) d \beta.
\end{equation}

System~\eqref{eqn:charODE}-\eqref{eqn:F-qgen} is a reformulation the PDE model~\eqref{ag2} in
terms of particle-trajectory equations.  Case $q \geq 2$ was studied in detail in~\cite{FeHuKo11}, by analogy 
with  the ODE setup of the incompressible Euler equations~\cite{MajdaBertozzi}.

Following \cite{MajdaBertozzi,FeHuKo11}, we consider the Banach space
\[
\mathcal{B} = \{ X : \mathbb{R}^{n} \to\mathbb{R}^{n} \text{ such that }
\|X\|_{1,\gamma} < \infty\},
\]
where $\|\cdot\|_{1,\gamma}$ is the norm defined by
\begin{equation}
\label{eqn:1gamma-norm}
\|X\|_{1,\gamma} = |X(0)| + \| \nabla_{\alpha}X
\|_{L^{\infty}} + |\nabla_{\alpha}X|_{\gamma}.
\end{equation}
Here, $|\cdot|$ is the H\"older seminorm
\[
|\nabla_{\alpha}X|_{\gamma}= \sup_{\alpha\neq\alpha^{^{\prime}}} \frac
{|\nabla_{\alpha}X(\alpha) - \nabla_{\alpha} X(\alpha^{\prime})|}{|\alpha-
\alpha^{\prime}|^{\gamma}}.
\]
Consider an open subset $\mathcal{O}_{L}$, of $\mathcal{B}$ defined by
\[
\mathcal{O}_{L} = \left\{  X \in\mathcal{B} \mid\inf_{\alpha}
\operatorname{det} \nabla_{\alpha}X(\alpha) > 1/L \text{ and } \|X\|_{1,\gamma
}<L \right\}  .
\]

The key ingredients to show local and global well-posedness of solutions are the properties of the convolution kernel 
\begin{equation}
\label{eqn:kr}
k(x)=\frac{1}{n\omega_{n}}\frac{x}{|x|^{n}},
\end{equation}
present in the repulsion component of~\eqref{eqn:f-exp}. In particular, $k$ is singular, homogeneous of degree $1-n$ and 
its gradient $P = \nabla k$ is homogeneous of degree $-n$ and defines a singular integral operator (SIO). The close 
analogy with incompressible fluid equations comes from the fact that a similar kernel appears in the Biot-Savart 
law~\cite{MajdaBertozzi}.

The local existence and uniqueness is stated by the following theorem.
\begin{thm}
\label{th:local} \textbf{(local existence and uniqueness)} Consider a
compactly supported initial density $\rho_{0} \in L^{\infty}(\mathbb{R}^{2})$,
with $|\rho_{0}|_{\gamma}<\infty$, for some $\gamma\in(0,1)$. Then for any
$L>0$, there exists $T(L)>0$ and a unique solution $X \in C^{1}((-T(L),T(L));
\mathcal{O}_{L})$ to~\eqref{eqn:charODE}-\eqref{eqn:F-qgen}, with $q>2-n$.
\end{thm}
\begin{proof}
Following~\cite{FeHuKo11, MajdaBertozzi}, we show that the map $\mathcal{F}$ is bounded and locally Lipschitz 
continuous on $\mathcal{O}_{L}$. Local existence and uniqueness then follows from Picard theorem. Details are 
presented in the Appendix.
\end{proof}

We now use a continuation result of solutions to autonomous ODE's on Banach
spaces (Theorem 4.4 in~\cite{MajdaBertozzi}) to upgrade the result to global
existence. Inspecting the set $\mathcal{O}_{L}$ we infer that we cease to have
a solution at a finite time $T_{*}$ provided either $\inf_{\alpha}
\operatorname{det} \nabla_{\alpha}X(\alpha)$ becomes $0$ or $\|X\|_{1,\gamma}$
becomes unbounded as $t \to T_{*}$.

The first scenario is ruled out by the following proposition.
\begin{prop}
\label{prop:J} At any fixed time $t<\infty$, solutions of~\eqref{eqn:charODE}-\eqref{eqn:F-qgen}  satisfy
\[
\inf_{\alpha} \operatorname{det} \nabla_{\alpha}X(\alpha,t) \geq e^{-C t},
\]
where $C>0$ is a constant.
\end{prop}
\begin{proof}
Use the differential equation that $J$ satisfies,
\[
\frac{d}{dt} J(\alpha,t) = J(\alpha,t) \nabla \cdot v(X(\alpha,t),t),
\]
to derive
\[
J(\alpha,t)=\exp\left( \int_{0}^{t} \nabla \cdot v(X(\alpha,s),s)d s\right).
\]
Case $q>2$ was considered in~\cite{FeHuKo11}. For $2-n<q<2$,  one can use~\eqref{eqn:div-vqneq2} and 
similar calculations to those leading to~\eqref{est:drho_dt}-\eqref{drho-max_dt}, to derive
\[
|\nabla \cdot v| \leq \| \rho \|_{L^\infty} + C \| \rho \|_{L^\infty}^{\frac{2-q}{n}}.
\]
As  $\| \rho \|_{L^\infty}$ is uniformly bounded in time (see Section~\ref{subsect:bounds}), we find 
\begin{equation*}
J(\alpha,t)   \geq \exp\left(  - C  t \right),
\end{equation*}
with $C>0$.
\end{proof}

The second scenario for the break-up of the solution (finite-time blow-up of
$\|X\|_{1,\gamma}$) will be treated as in Chapter 4.2 of~\cite{MajdaBertozzi}. This procedure was used in~\cite{FeHuKo11} to 
study global well-posedness of solutions to~\eqref{eqn:charODE}, \eqref{eqn:F-qgen}
when $q \geq 2$. In summary, $\|X\|_{1,\gamma}$ can be shown to remain bounded for all finite times,  
provided $\int_0^t \| \rho(\cdot,s) ds\| ds<\infty$, for all $t$. This is an analogue of the Beale-Kato-Majda 
condition for global existence of solutions to incompressible Euler equations~\cite{BealeKatoMajda, MajdaBertozzi}. Our 
argument is adapted from the analysis of the incompressible fluid equations presented in Chapter 4 of~\cite{MajdaBertozzi}.

A first a priori bound is provided by the following proposition.
\begin{prop}
\label{prop:1gamma-norm} Provided $\int_{0}^{t} \| \nabla v (\cdot,
s)\|_{L^{\infty}}$ has an a priori bound, $\| \nabla_{\alpha}X(\cdot
,t)\|_{L^{\infty}}$ and $| \nabla_{\alpha}X(\cdot,t)|_{\gamma}$ are a priori bounded.
\end{prop}
\begin{proof} The proof for $q=2$ was presented in detail in~\cite{FeHuKo11} (see Proposition 2.3). Extending 
it to general $q>2-n$ does not pose any difficulties and we omit the details. See also Proposition 4.3 
from~\cite{MajdaBertozzi} for the corresponding result in the context of incompressible Euler equations.
\end{proof}

The Beale-Kato-Majda condition for global existence of incompressible Euler equations is an a priori control on the
time-integral of the supremum norm of vorticity. In the context of our
aggregation model~\eqref{eqn:charODE}-\eqref{eqn:F-qgen}, this condition will be
replaced by an a priori bound on $\int_{0}^{t} \left\|  \rho(\cdot,s)
\right\|  _{L^{\infty}} ds$.

\begin{prop}
\label{lemma:BKM} A sufficient condition for $\int_{0}^{t} \| \nabla v (\cdot,
s)\|_{L^{\infty}} ds $ to be a priori bounded is an a apriori bound on
$\int_{0}^{t} \| \rho(\cdot, s)\|_{L^{\infty}} ds $. More specifically,
\[
\int_{0}^{t} \| \nabla v (\cdot, s)\|_{L^{\infty}} ds \leq e^{C(\rho_{0}%
)\int_{0}^{t} \| \rho(\cdot, s)\|_{L^{\infty}} ds },
\]
where $C(\rho_{0})$ is a constant that depends on the initial density only.
\end{prop}
\begin{proof}
Case $q \geq 2$ was discussed in~\cite{FeHuKo11}. The proof for $2-n<q<2$ requires some 
slight adaptations from the corresponding result for fluids (see Theorem 4.3 and its proof in~\cite{MajdaBertozzi}). 
The repulsion component of $ \nabla v $ can be estimated using \eqref{eqn:estimate2}. The attraction part has a milder 
singularity and does not break the estimate, hence we have
\[
\| \nabla v (\cdot,t)\|_{L^\infty} \leq  c \left (  |\rho(\cdot,t)|_\gamma \epsilon^\gamma + \| 
\rho (\cdot,t)\|_{L^\infty} \log (1/\epsilon) \right) +M.
\]
Set $\epsilon=  |\rho(\cdot,t)|_\gamma^{-1/\gamma}$ to get
\[
\| \nabla v (\cdot,t)\|_{L^\infty} \leq \| \rho (\cdot,t)\|_{L^\infty} \left ( \log  |\rho(\cdot,t)|_\gamma + c \right).
\]
Lemma 4.8 in~\cite{MajdaBertozzi} can be trivially adapted to our context, resulting in the following inequality:
\[
| \rho(\cdot,t)|_{\gamma} \leq |\rho_0|_{\gamma} \exp \left( (c + \gamma) \int_0^t \| \nabla v (\cdot, s)\|_{L^\infty} ds\right).\]
By combining the last two inequalities we find
\[
\| \nabla v (\cdot,t)\|_{L^\infty} \leq C(\rho_0) \| \rho (\cdot,t)\|_{L^\infty} \left( 1 +  \int_0^t \| 
\nabla v (\cdot, s)\|_{L^\infty} ds\right).
\]
The desired inequality follows after division by $1 +  \int_0^t \| \nabla v (\cdot, s)\|_{L^\infty} ds$ and integration 
with respect to $t$.
\end{proof}

Finally, we have all the ingredients to prove global existence of solutions.
The result is given by the following theorem.
\begin{thm}
\label{thm:prop-genq} \textbf{(global existence and uniqueness)} Consider the trajectory equations~\eqref{eqn:charODE}, 
\eqref{eqn:F-qgen} with the Banach space setup and notations as above, and a
compactly supported initial density $\rho_{0} \in L^{\infty}(\mathbb{R}^{n})$,
with $|\rho_{0}|_{\gamma}<\infty$, for some $\gamma\in(0,1)$. Then, for every
$T$, there exists $L>0$ and a unique solution $X \in C^{1}([0,T);
\mathcal{O}_{L})$ to~\eqref{eqn:charODE}, \eqref{eqn:F-qgen} with $q>2-n$ (a unique
solution exists globally in time).
\end{thm}
\begin{proof}
Case $q \geq 2$ was studied in~\cite{FeHuKo11} and we focus here on the range $ 2-n<q<2$.

The solution $X$ is in the set $\mathcal{O}_L$ provided
\begin{equation}
\label{eqn:2cond}
\inf_{\alpha} \operatorname{det} \nabla_\alpha X(\alpha) > 1/L \quad \textrm{ and } \quad   \|X\|_{1,\gamma}<L.
\end{equation}
Using Proposition~\ref{prop:J}, the first condition is satisfied provided we choose $L>e^{CT}$. We now investigate 
the second condition in~\eqref{eqn:2cond}. Start by inspecting the first term in~\eqref{eqn:1gamma-norm}, $|X(0,t)|$. 
Integrate~\eqref{eqn:char} to get
\begin{equation}
\label{eqn:X0}
X(0,t) =  \int_0^t v (X(0,s),s)  ds.
\end{equation}
The repulsion component of $v$ can be bounded in terms of  $\| \rho \|_{L^\infty}$ using~\eqref{eqn:estimate1}. 
The attraction component also has a uniform bound. To show this  we distinguish two 
cases: (i) $2-n<q<1$ and (ii) $1 \leq q <2$, and inspect the attraction part of $v$, i.e., $-|x|^{q-2} x \ast \rho$. In case (i),
\begin{align*}
\left | |x|^{q-2} x \ast \rho \right | &\leq \int_{|x-y|<1} \frac{1}{|x-y|^{1-q}} \rho(y) dy + \int_{|x-y|>1} 
\frac{1}{|x-y|^{1-q}} \rho(y) dy \\
&\leq  c \| \rho \|_{L^\infty }  + M.
\end{align*}
In case (ii) solutions $\rho$ are compactly supported (uniformly in time) in a ball of radius $R$, provided the 
initial density $\rho_0$ is --- see Remark in Section~\ref{subsect:bounds}. Hence,
\[
\left | |x|^{q-2} x \ast \rho \right | \leq (2R)^{q-1} M.
\]
Using~\eqref{eqn:X0} we derive
\begin{equation}
\label{eqn:estX0}
|X(0,t)| \leq c_1  \int_0^t  \| \rho(\cdot,s) \|_{L^\infty} ds + c_2 t,
\end{equation}
The control of  $\|X\|_{1,\gamma}$ now follows from~\eqref{eqn:estX0}, Propositions~\ref{prop:1gamma-norm} and~\ref{lemma:BKM}.  
More precisely, by redefining the constants appropriately, one can derive
\begin{equation}
\label{eqn:estX-1gamma}
\|X(\cdot,t) \|_{1,\gamma} \leq C e^{c_1t} e^{c_2 e^{c_3  \int_0^t  {\| \rho(\cdot,s) \|_{L^\infty}} ds}}.
\end{equation}
Given that $\| \rho(\cdot,s) \|_{L^\infty} $ is uniformly bounded, we can choose the constant $L$ large enough such 
that  $\|X(\cdot,t)\|_{1,\gamma}<L$, for all $t \in [0,T)$.

\end{proof}


\section{Radially symmetric steady states}
\label{sect:sstates}

In this section we show that the aggregation model admits a unique radially symmetric
steady state $\bar{\rho}$ supported on a ball of  $\mathbb{R}^n$. For $2-n<q<2$, we prove that these steady states are monotonically decreasing about the origin, while for $q>2$ they are increasing about it. Case $q=2$ corresponds
 to constant solutions in a ball \cite{FeHuKo11}. We further study these equilibria using numerical and asymptotic methods in Section \ref{sect:asy-num}.


\subsection{Existence and uniqueness of equilibria supported on a ball}
\label{subsect:exist}

Suppose that $\bar{\rho}$ is a steady state with support the  ball $B(0,R)$ centred at the origin, of radius $R$. 
The velocity $v$ is zero in $B(0,R)$, so its divergence also vanishes. Hence, from~\eqref{eqn:div-vqneq2} we find 
that $\bar{\rho}$ satisfies the following integral equation,
\begin{equation}
\label{eqn:int-eq}
\bar{\rho}(x) - (n+q-2)\int_{\mathbb{R}^{n}}|x-y|^{q-2}\bar{\rho}(y)dy =0,
\qquad\text{ for } x \in B(0,R),
\end{equation}
and vanishes outside $B(0,R)$.

Consider  the operator $T_R$ given by
\begin{equation}
\label{eqn:TR}
T_R \bar{\rho} (x) = (n+q-2)\int_{B(0,R)}|x-y|^{q-2}\bar{\rho}(y)dy.
\end{equation}
The subscript is used to emphasize the dependence of the integral operator on the radius $R$. 
As $q-2 > -n$, the kernel $|x-y|^{q-2}$ is integrable and $T_R$ defines a linear bounded operator from 
$C(B(0,R), \mathbb{R})$ to itself. Equation~\eqref{eqn:int-eq} can be cast as an eigenvalue problem,
\begin{equation*}
T_R \bar{\rho} = \bar{\rho}, \qquad  \bar{\rho} \in C(B(0,R), \mathbb{R}),
\end{equation*}
where solutions $\bar{\rho}$ are eigenfunctions corresponding to eigenvalue $1$.

The existence and uniqueness of a steady state supported in $B(0,R)$ is provided by the following theorem.

\begin{thm}
\label{thm:PF} For every $q>2-n$ and  $M>0$, there exists a unique radius $R$ (that
depends on $q$ and $n$ only) and a unique steady state
$\bar{\rho}$ of the aggregation model \eqref{eq:aggre}-\eqref{eq:kernel} that is supported on $B(0,R)$, has 
mass $M$ and is continuous on its support.
\end{thm}
\begin{proof}
 We use a scaling argument and consider the case $R=1$ first. For $q>2-n$,  the kernel $|x-y|^{q-2}$ is 
integrable and the linear operator $T_1: \left(C(B(0,1), \mathbb{R}),\| \cdot \|_{L^\infty} \right) \to \left( C(B(0,1), \mathbb{R}), \| \cdot \|_{L^\infty} \right)$ is bounded.  The operator  is also compact. This is a textbook exercise in analysis~\cite{Vainikko1993}, but we include it here for 
completeness. Case $q \geq 2$ presents no difficulty, as the kernel $|x-y|^{q-2}$ is continuous. For $2-n<q<2$, 
we write $T_1$ as a limit of compact operators as follows. Consider a smooth cut-off function $h :[0,\infty) \to \mathbb{R}$, 
such as $h(r)=0$ for $0\leq r \leq 1/2$, $h(r) = 1$ for $r \geq 1$ and $0 \leq h(r) \leq 1$ for all $r \geq 0$. Define
 \[
 K_p(x,y) = h(p |x-y|) |x-y|^{q-2},
 \]
 and 
 \[
 T_{1p}\bar{\rho} (x) = (n+q-2)  \int_{B(0,1)} K_p (x,y) \bar{\rho}(y)dy.
 \]
 The kernels $K_p$ are continuous, hence the operators $T_{1p}$ are compact. Now estimate
 \begin{align*}
 |T_{1p}(x)-T_1(x)| &\leq  (n+q-2) \int_{B(0,1)} |x-y|^{q-2} | 1 - h(p |x-y|)| \bar{\rho}(y) dy \\
 & \leq (n+q-2) \| \bar{\rho} \|_{L^\infty} \int_{|x-y|<1/p} |x-y|^{q-2} dy \\
 & \leq n \omega_n  \| \bar{\rho} \|_{L^\infty} \frac{1}{p^{n+q-2}}.
 \end{align*}
 As $n+q-2>0$, $\frac{1}{p^{n+q-2}} \to 0$, when $p \to \infty$. The convergence is uniform in $x$, hence,
 \[
 \| T_{1p}-T_1\| \to 0, \qquad p \to \infty,
 \]
 and $T_1$ is compact as a limit (in the operator norm) of compact operators.
 
 We apply the Krein-Rutman theorem~\cite{Du2006} to operator $T_1$ in the following setup. Take the cone 
in $C(B(0,1),\mathbb{R})$ consisting of all non-negative functions. $T_1$ is a linear, strongly positive, 
compact operator that maps the space of continuous functions $C(B(0,1),\mathbb{R})$ into itself. By 
Krein-Rutman theorem (see Theorem 1.2 in~\cite{Du2006}), there exists a  positive eigenfunction 
$\bar{\rho}_1$ such that $T_1\bar{\rho}_1=\lambda\bar{\rho}_1$, where
$\lambda$ (which depends only on $q$ and $n$) is the spectral radius of $T_1$.  Moreover, the 
eigenvalue $\lambda$  is simple and there is no other eigenvalue with a positive eigenvector.
By making the change of variable
\begin{equation}
\label{eqn:rho-scale}\bar{\rho}(x)= \bar{\rho}_{1}(x/R)
\end{equation}
in~\eqref{eqn:TR}, we get
\[
T_{R}\bar{\rho}(x)=R^{n+q-2}\lambda\, \bar{\rho}(x).
\]
Now ask that $\bar{\rho}$ is an eigenfunction of $T_{R}$ corresponding to
eigenvalue one and find
\begin{equation}
\label{eqn:R-gen}R=\lambda^{-\frac{1}{n+q-2}},
\end{equation}
which gives the radius of the support as a function of $q$ and $n$ only. Once a mass $M$ for $\bar{\rho}$ 
is set, uniqueness can be inferred from the uniqueness properties of the spectral radius of $T_1$ and its 
associated eigenfunction $\bar{\rho}_1$.
\end{proof}


\subsection{Qualitative properties of equilibria}
\label{subsect:monotonicity}

\paragraph{Monotonicity and radial symmetry of equilibria.}
We prove that the equilibria given by positive solutions of~\eqref{eqn:int-eq}  (and whose existence and 
uniqueness was established in Theorem~\ref{thm:PF})  are radially symmetric and monotone in the radial 
coordinate. More specifically, equilibria are monotonically decreasing when $2-n<q<2$ and increasing 
for $q>2$. To prove this result we employ the {\em method of moving planes}, a technique introduced by 
the Soviet mathematician Alexandroff in the early 1950's, which became well-known after 
Gidas, Ni and Nirenberg~\cite{GidasNiNirenberg} applied it to study qualitative properties of positive 
solutions of elliptic equations. We refer to~\cite{Du2006} for a detailed description of the method and its applications.

Our use of the moving plane technique is inspired by a novel application of the method in the context of integral 
equations~\cite{ChenLiOu}. Consider a steady state $\bar{\rho}$ supported on the ball $B(0,R)$; $\bar{\rho}$ satisfies 
the integral equation~\eqref{eqn:int-eq} in $B(0,R)$ and vanishes outside $B(0,R)$. For convenience of calculations, denote 
\[
\alpha = n+q-2.
\]
As $q>2-n$, we have $\alpha>0$. Two cases will be distinguished from the subsequent analysis:  
(i) $2-n<q<2$ (equivalently, $0<\alpha<n$) with radial equilibria which decrease about the origin, 
and (ii) $q>2$ (or $\alpha>n$) with increasing solutions. Using the new notation, $\bar{\rho}$ satisfies
\begin{equation}
\label{eq:rhobar}
\bar{\rho}(x)=\begin{cases}
\alpha \int _{B(0,R)} \frac{1}{|x-y|^{n-\alpha}} \bar{\rho}(y) \, dy  & x \in B(0,R)\\
0 & x \notin B(0,R). 
\end{cases}
\end{equation}

Take $\mu \in \mathbb{R}$ such that $-R<\mu<0$, and consider the reflection across the 
plane $x_1= \mu$, $x \mapsto x^\mu = (2 \mu - x_1, x_2, \dots, x_n)$. In particular, we have 
the image $0^\mu$ of the origin $0$ under this map. Define 
\[
\bar{\rho}_\mu(x) = \bar{\rho}(x^\mu).
\]
Using~\eqref{eq:rhobar}, $\bar{\rho}_\mu$ is given by
\begin{equation}
\label{eq:rho-lambda}
\bar{\rho}_\mu(x)=\begin{cases}
\alpha \int _{B(0,R)} \frac{1}{|x^\mu-y|^{n-\alpha}} \bar{\rho}(y) \, dy & x \in B(0^\mu,R)\\
0 & x \notin B(0^\mu,R). 
\end{cases}
\end{equation}
Define
\[
\Sigma_\mu = \{ x \in B(0,R) \mid x_1 \geq \mu \}.
\]
We apply the method of moving planes and compare $\bar{\rho}(x)$ and $\bar{\rho}_\mu(x)$ for $x \in \Sigma_\mu$. 
In case (i),  $2-n<q<2$, we show that there is a $\mu$, $-R<\mu<0$, such that $\bar{\rho}(x) \geq \bar{\rho}_\mu(x)$, 
for all $x \in \Sigma_\mu$. By a  continuity argument, we  show that the plane $x_1=\mu$ can be moved continuously all 
the way to $x_1=0$, and hence, $\bar{\rho}(x)$ increases as $x$ approaches the origin from $x_1<0$. A similar argument can be 
made using planes $x_1=\mu>0$. Since the direction $x_1$ can be chosen arbitrarily we conclude that $\bar{\rho}$ is radially 
symmetric and decreasing about the origin. For case (ii), $q>2$, a similar argument leads to  $\bar{\rho}$ being radial and 
increasing about the origin.

We now state and prove the result regarding the monotonicity of equilibria of~\eqref{eq:aggre}-\eqref{eq:kernel}.
\begin{thm}
\label{th:monotone}
Consider a bounded steady state $\bar{\rho}(x)$ of the aggregation model~\eqref{eq:aggre}-\eqref{eq:kernel} that 
is supported in a ball $B(0,R)$ of $\mathbb{R}^n$. Then, $\bar{\rho}$ is radially symmetric and monotone about the origin. 
More specifically, we distinguish two cases: (i) $2-n<q<2$, when $\bar{\rho}$ is decreasing about the origin, and 
(ii) $q>2$, when $\bar{\rho}$ is increasing. 
\end{thm}

\begin{proof}
We use the notations and symbols introduced in the preamble of the theorem. Calculate $\bar{\rho}(x) - \bar{\rho}_\mu(x)$, 
for $x \in \Sigma_\mu$. As $\bar{\rho}_\mu(x) = 0$ outside $B(0^\mu,R)$, we have
\[
\bar{\rho}(x) - \bar{\rho}_\mu(x) = \bar{\rho}(x) >0 , \qquad \text{ for } x \in \Sigma_\mu \setminus B(0^\mu,R).
\]
It remains to consider the case  $x \in  \Sigma_\mu \cap B(0^\mu,R)$. Denote by $\Sigma_\mu^C$ the complement 
of $\Sigma_\mu$ in $B(0,R)$, i.e., 
\[
\Sigma_\mu^C = B(0,R) \setminus \Sigma_\mu,
\]
and calculate, using \eqref{eq:rhobar},
\begin{align*}
\bar{\rho}(x) &= \alpha \int_{\Sigma_\mu}  \frac{1}{|x-y|^{n-\alpha}} \bar{\rho}(y) \, dy + \alpha 
\int_{\Sigma_\mu^C}  \frac{1}{|x-y|^{n-\alpha}} \bar{\rho}(y) \, dy \\
&= \alpha \int_{\Sigma_\mu}  \frac{1}{|x-y|^{n-\alpha}} \bar{\rho}(y) \, dy + \alpha \int_{\Sigma_\mu \cap 
B(0^\mu,R)}  \frac{1}{|x^\mu-y|^{n-\alpha}} \bar{\rho}_{\mu}(y) \, dy.
\end{align*}
In the above calculation we used $|x^\mu-y^\mu| = |x-y|$ to write the second integral in the right-hand-side 
as an integral with respect to $y^\mu$ (subsequently relabelled $y$).

Similarly, using \eqref{eq:rho-lambda},
\begin{align*}
\bar{\rho}_\mu(x) 
&= \alpha \int_{\Sigma_\mu}  \frac{1}{|x^\mu-y|^{n-\alpha}} \bar{\rho}(y) \, dy + \alpha \int_{\Sigma_\mu 
\cap B(0^\mu,R)}  \frac{1}{|x-y|^{n-\alpha}} \bar{\rho}_{\mu}(y) \, dy.
\end{align*}
As $\bar{\rho}_\mu$ is zero outside $B(0^\mu,R)$, the second integrals in the right-hand-sides of the 
expressions for $\bar{\rho}$ and $\bar{\rho}_\mu$ above can be extended to $\Sigma_\mu$. Hence, we compute, 
for $x \in \Sigma_\mu \cap B(0^\mu,R)$,
\begin{equation}
\label{eqn:diff}
\bar{\rho}(x) - \bar{\rho}_\mu(x) = \alpha \int_{\Sigma_\mu} \left( \frac{1}{|x-y|^{n-\alpha}}  - 
\frac{1}{|x^\mu-y|^{n-\alpha}} \right) (\bar{\rho}(y) - \bar{\rho}_\mu(y) ) \, dy.
\end{equation}

{\em Case (i)} $2-n<q<2$ (or $0<\alpha<n$): For $x \in \Sigma_\mu \cap B(0^\mu,R)$ and $y \in \Sigma_\mu$, 
$|x-y| < |x^\mu - y|$. Hence, 
\begin{equation}
\label{eqn:ineq}
\frac{1}{|x-y|^{n-\alpha}}  - \frac{1}{|x^\mu-y|^{n-\alpha}}  \geq 0, \qquad \text{ for all } x \in \Sigma_\mu 
\cap B(0^\mu,R) \text{ and } y \in \Sigma_\mu.
\end{equation}
Let us first assume that there exists $\mu_0 \in (-R, 0)$ such that 
\begin{equation}
\label{eqn:lambda0}
\bar{\rho}(x) \geq \bar{\rho}_{\mu_0}(x), \qquad  \text{ for all } x \in \Sigma_{\mu_0},
\end{equation}
and prove that $\mu_0$ can be extended all the way to the origin $0$. Suppose by contradiction that $\mu_0$ cannot 
be extended. Take $\mu \in (\mu_0, \mu_0+\epsilon)$, $\epsilon>0$, and define
\[
\Sigma_{\mu}^- = \{ x \in \Sigma_{\mu} \mid \bar{\rho}(x) <  \bar{\rho}_\mu(x)\}.
\]
For $x \in \Sigma_{\mu}^-$, using~\eqref{eqn:diff} and \eqref{eqn:ineq} we find
\[
0< \bar{\rho}_\mu(x) - \bar{\rho}(x) \leq \alpha \int_{\Sigma_{\mu}^-}  \left( \frac{1}{|x-y|^{n-\alpha}}  
- \frac{1}{|x^\mu-y|^{n-\alpha}} \right) (\bar{\rho}_\mu(y) - \bar{\rho}(y) ) \, dy,
\]
and hence,
\begin{equation}
\label{eqn:est-Linf}
\|  \bar{\rho}_\mu - \bar{\rho} \|_{{L^\infty}(\Sigma_{\mu}^-)} \leq \alpha \|  \bar{\rho}_\mu 
- \bar{\rho} \|_{{L^\infty}(\Sigma_{\mu}^-)} \sup_{x \in \Sigma_{\mu}^- }\int_{\Sigma_{\mu}^-}  
\left( \frac{1}{|x-y|^{n-\alpha}}  - \frac{1}{|x^\mu-y|^{n-\alpha}} \right) \, dy,
\end{equation}
The function under the integral on the right-hand-side is integrable, as $\alpha>0$. Also, from~\eqref{eqn:diff}, 
we infer that $\bar{\rho}(x) > \bar{\rho}_\mu(x)$ in the interior of $\Sigma_{\mu_0}$, which implies that the
 closure $\overline{\Sigma_{\mu_0}^-}$ of $\Sigma_{\mu_0}^-$ has measure $0$. As $\lim_{\mu \to \mu_0} 
\Sigma_{\mu}^- \subset \overline{\Sigma_{\mu_0}^-}$, we conclude that the measure of $\Sigma_{\mu}^-$ approaches 
$0$ as $\mu \to  \mu_0$. Therefore, we can choose $\epsilon$ small enough such that
\[
\alpha \sup_{x \in \Sigma_{\mu}^- }\int_{\Sigma_{\mu}^-}  \left( \frac{1}{|x-y|^{n-\alpha}}  -
 \frac{1}{|x^\mu-y|^{n-\alpha}} \right) \, dy \leq \frac{1}{2}.
\]
By \eqref{eqn:est-Linf} we have $ \|  \bar{\rho}_\mu - \bar{\rho} \|_{{L^\infty}(\Sigma_{\mu}^-)} $ =0, 
which implies that $\Sigma_{\mu}^-$ is empty, hence $\mu_0$ is not maximal and we reached the desired contradiction.

It remains to show that $\mu_0 \in (-R,0)$ with the property~\eqref{eqn:lambda0} exists indeed. Note 
that~\eqref{eqn:lambda0} holds for $\mu_0 = -R$ (see~\eqref{eq:rhobar}  and~\eqref{eq:rho-lambda}). An argument 
entirely similar to the one made in the previous step proves that the plane  $\mu_0 = -R$ can be moved further to 
the right, while~\eqref{eqn:lambda0} still holds. This concludes the proof for case (i).
\medskip

{\em Case (ii)} $2<q$ (or $\alpha>n$) follows similarly, the single difference being the sign of the integrand 
in~\eqref{eqn:diff}. More precisely, instead of~\eqref{eqn:ineq}, the reversed inequality holds:
\begin{equation*}
\frac{1}{|x-y|^{n-\alpha}}  - \frac{1}{|x^\mu-y|^{n-\alpha}}  \leq 0, \qquad \text{ for all } x \in \Sigma_\mu 
\cap B(0^\mu,R) \text{ and } y \in \Sigma_\mu.
\end{equation*}
A similar argument as that made for case (i) shows that $\bar{\rho}$ decreases as $x$ approaches the origin 
from $x_1<0$. Hence conclude that in this case  $\bar{\rho}$ is radially symmetric and increasing about the origin.
\end{proof}

\smallskip
{\bf Remarks} 

a. Case $q=2$ corresponds to constant solutions in a ball. These solutions were studied in detail in~\cite{FeHuKo11}.

b. Theoretical findings are in perfect agreement with numerical results --- see Section~\ref{sect:asy-num}.

c. The convexity of the steady state can be inferred easily from~\eqref{eqn:int-eq} when $q>3$ (the 
function $|x|^\alpha$ is convex for $\alpha>1$).




\section{Numerical and asymptotic studies of radial equilibria}
\label{sect:asy-num}

We showed that a steady state $\bar{\rho}$ supported on a ball is necessarily radially symmetric and monotone. 
In this section we investigate further their properties using numerical and asymptotic methods.

\subsection{Numerical methods for the dynamic evolution and steady states}
\label{subsect:numerics}

In solving numerically  the steady states~\eqref{eqn:int-eq} or the dynamic evolution of solutions 
to~\eqref{eq:aggre}-\eqref{eq:kernel}, the computational bottleneck is the evaluation of the integral operators. The methods we 
use are similar to those in~\cite{HuBe2010,FeHuKo11} and are reviewed and extended below. 

\paragraph{Steady states.} The equilibria are computed from~\eqref{eqn:int-eq} by using the power 
method~\cite{EastmanEstep}, whose convergence is guaranteed by 
 Theorem~\ref{thm:PF}.  First,  write the operator $T_R$ given by~\eqref{eqn:TR} in radial coordinates:
\begin{equation}
\label{eqn:T_R}T_{R}\bar{\rho}(r)=\frac{n(n+q-2)\omega_{n}}{\int_{0}^{\pi}\sin^{n-2}\theta d\theta}
\int_{0}^{R} (r{^{\prime}})^{n-1}%
\bar{\rho}(r{^{\prime}})I(r,r{^{\prime}})dr{^{\prime}},
\end{equation}
where
\begin{equation}
\label{eqn:I}
I(r,r{^{\prime}})=\int_{0}^{\pi}(r^{2}+(r{^{\prime}})^{2}-2rr{^{\prime}}%
\cos\theta)^{q/2-1}\sin^{n-2}\theta d\theta.
\end{equation}
The steady states $\bar{\rho}$ are eigenfunctions of  $T_{R}$ corresponding to eigenvalue $1$. Here the eigenvalue
problem consists in determining the eigenfunction $\bar{\rho}$ and the radius
$R$ of the support. The actual steady density is a constant multiple of this
eigenfunction, where the constant is determined from the initial mass. 

By the scaling argument used to prove Theorem~\ref{thm:PF} , it is enough to find the spectral radius $\lambda$ of $T_1$, 
and its corresponding eigenfunction $\bar{\rho}_1$; the steady state $\bar{\rho}$ can then be calculated from~\eqref{eqn:R-gen} 
and~\eqref{eqn:rho-scale}.

Given an initial positive density $\bar{\rho}^{(0)}$ on $[0,1]$, consider the iterative scheme (power method):
\begin{equation}\label{eqn:powermethod}
\bar{\rho}^{(m+1)}= T_{1}\bar{\rho}^{(m)}/\|T_{1}\bar{\rho}^{(m)}\|,\quad \lambda^{(m)} = 
\|T_{1} \bar{\rho}^{(m)}\|/\|{\bar{\rho}}^{(m)}\|,
\end{equation}
where $\|\cdot\|$ can be any norm for functions on the unit interval $[0,1]$.
The sequence $\bar{\rho}^{(m)}$ converges to $\bar{\rho}_{1}$~\cite{EastmanEstep}, and the
spectral radius of $T_{1}$ is given by the limit of $\lambda^{(m)}$, as $m$ goes to infinity. 

The integral operator in~\eqref{eqn:T_R} is discretized by the trapezoidal rule. The computational complexity can be 
reduced to the calculation of the integration in $r'$ only, where $I(r,r')$ can be interpolated using the fact that
$I(r,r')=\max(r,r')^{q-2}I_1(s)$, where  
\begin{equation}
\label{eqn:I1}
 I_{1}(s)    = \int_{0}^{\pi}(1+s^{2}-2s\cos\theta)^{q/2-1}\sin
^{n-2}\theta d\theta,\qquad s=\min(r,r')/\max(r,r') \in [0,1],
\end{equation}
and $I_1$ is computed at sample points on $[0,1]$. 

The calculations do not present significant challenges, except for $q \in (2-n,3-n]$, when $I_1(1)$ is unbounded, making the error in the trapezoidal rule uncontrollable.
We managed to calculate this more singular case only in  dimensions one and three, where $I(r,r')$ 
can be obtained explicitly. Due to these computational difficulties, the singular limit  $q \to 2-n$ studied below by asymptotic methods is valid in  all dimensions, but its numerical verification is done only in one and three dimensions.

In three dimensions for instance, the kernel $I$ from \eqref{eqn:I} can be calculated explicitly: when $r,r'\neq 0$,
\begin{align}
I(r,r')
    =\begin{cases}
        \frac{(r+r')^{q}-|r-r'|^q}{qrr'},\qquad & \mbox{if}\ q \neq 0\cr
        \frac{1}{rr'}\ln\frac{r+r'}{|r-r'|},& \mbox{if}\ q=0,
    \end{cases}
    \label{eqn:explicitI}
\end{align} 
and when $r'=0$ or $r=0$,
\[
 I(r,0)= 2r^{q-2},\quad  I(0,r')=2r'^{q-2}.
\]
Using the explicit form~\eqref{eqn:explicitI}, when $q$ is in the singular range $(-1,0]$, 
and $r\approx r'$, the integrand $r'^2I(r,r')\rho(r')$ in the integral operator~\eqref{eqn:T_R} is 
weakly singular. However, the part associated with this weak singularity can be approximated as the product 
of the weakly singular function $|r-r'|^{-q}$ and a smooth function. When the latter is approximated by 
linear or higher order interpolation, the whole integral can be calculated explicitly\cite{Egger}, giving
an accurate approximation of~\eqref{eqn:T_R}.  


The steady states have been displayed in Figure~\ref{fig:variousq} in dimension three, for a wide range of values of $q$. 
The results are perfectly consistent with the monotonicity properties stated and proved in Theorem~\ref{th:monotone}. 
As noted from Figure~\ref{fig:variousq}, equilibria display an interesting asymptotic behaviour 
 as $q \to \infty$ and $q \searrow 2-n$. As $q \to \infty$, the radii of the equilibria approach a constant, 
and mass aggregates toward the edge of the swarm. As $q \searrow 2-n$ mass concentrates at the origin, as 
attraction becomes as strong as the Newtonian repulsion. We perform careful asymptotic studies in Section~\ref{subsect:asympt}  and  confirm and detail these observations.


\paragraph{Dynamic evolution to equilibria.}  
All numerical simulations we performed suggest that equilibria that solve \eqref{eqn:int-eq} are global attractors for the aggregation model  \eqref{eq:aggre}-\eqref{eq:kernel}. 
To evolve dynamically the solutions to~\eqref{eq:aggre}-\eqref{eq:kernel}, we consider the model in characteristic 
form  (see~\eqref{eqn:charODE}, \eqref{eqn:F-qgen}, \eqref{eqn:rho-evol}). In radial coordinates the characteristic 
equations read
\begin{subequations}
\label{eq:charr}%
\begin{align}
\frac{dr}{dt}  &  = \frac{1}{r^{n-1}}\int_{0}^{r} (r{^{\prime}})^{n-1}\rho(r{^{\prime}%
})dr{^{\prime}}
\nonumber\\
&  \quad-\omega_{n-1}\int_{0}^{\infty}(r{^{\prime}})^{n-1}\rho(r{^{\prime}})\int_{0}^{\pi
}(r-r{^{\prime}}\cos\theta)(r^{2}+r{^{\prime}}^{2}-2rr{^{\prime}}\cos
\theta)^{q/2-1}\sin^{n-2} \theta d\theta dr{^{\prime}}\label{eq:drdt}\\
\frac{d\rho}{dt}  &  =-\rho\left[\rho-  (n+q-2)\omega_{n-1}\int_{0}^{\infty}%
(r{^{\prime}})^{n-1}\rho(r{^{\prime}})\int_{0}^{\pi}(r^{2}+r{^{\prime}}^{2}-2rr{^{\prime}}%
\cos\theta)^{q/2-1}\sin^{n-2}\theta d\theta dr{^{\prime}}\right]  .
\label{eq:drhodt}%
\end{align}
\end{subequations}
The term associated with the singular repulsion (the Newtonian potential $\phi$)  
in the right-hand-side of~\eqref{eq:drdt} is  calculated by taking advantage of the 
fact that the corresponding kernel is the fundamental solution of the Laplace equation. 


Similar to the discretization of the integral from \eqref{eqn:T_R}, the computational complexity in calculating the integrals from~\eqref{eq:drdt} and~\eqref{eq:drhodt} can be reduced by introducing the following auxiliary functions: $I_1$ (defined by \eqref{eqn:I1}),
\begin{subequations}%
\begin{align*}
I_{2}(s)  &  = \int_{0}^{\pi}(1-s\cos\theta)(1+s^{2}-2s\cos\theta
)^{q/2-1}\sin^{n-2}\theta d\theta,\\
I_{3}(s)  &  = \int_{0}^{\pi}(s-\cos\theta)(1+s^{2}-2s\cos\theta
)^{q/2-1}\sin^{n-2}\theta d\theta.
\end{align*}
\end{subequations}
Thus, the angular integrals in $\theta$ in~\eqref{eq:drdt} and~\eqref{eq:drhodt} become products of powers of $r$, $r{^{\prime}}$, and these auxiliary
functions, with $s=\min(r,r{^{\prime}})/\max(r,r{^{\prime}})$. Hence the double
integrals in \eqref{eq:drdt} and~\eqref{eq:drhodt}  become single integrals in $r^{\prime}$ and are
evaluated by trapezoidal rule.  Due to the extra factor $1-\cos\theta$ in the integrand, $I_{2}(1)$ and $I_{3}(1)$ are bounded for any $q>2-n$, and can always be used in the trapezoidal rule. 
This observation reduces the total complexity
in the computations  to $O(N^{2})$ per time step, where $N$ is the number
of spatial gridpoints in the radial coordinate $r$. Once the characteristic speeds in~\eqref{eq:charr}
are found, the equations are evolved in time by the classical fourth order
Runge-Kutta method.




Figures~\ref{fig:dynamic}(a) and~\ref{fig:dynamic}(b) show simulation results in three
dimensions, corresponding to $q=1.5$ and $q=20$, respectively. We plot the
solution against the radial coordinate $r$. The initial data used in Figure~\ref{fig:dynamic} is
\begin{equation}\label{eq:initrho}
\rho(x,0) = (0.2-20|x|^2+1000|x|^4)\exp(-40|x|^2)/c,
\end{equation}
where $c$ is a constant chosen to normalize the mass to one. The solutions approach as $t \to \infty$ the 
steady states studied in Section~\ref{sect:sstates} and shown in Figure~\ref{fig:variousq}. We note that for large $q$, the convergence near the origin tends to be slow. Based on numerical 
observations we conjecture that these equilibria  are {\em global attractors} for solutions 
to~\eqref{eq:aggre}-\eqref{eq:kernel}. In future work we plan to validate rigorously these observations 
regarding the global stability of the steady states.

\begin{figure}[htb]
 \begin{center}
\includegraphics[width=0.46\textwidth]{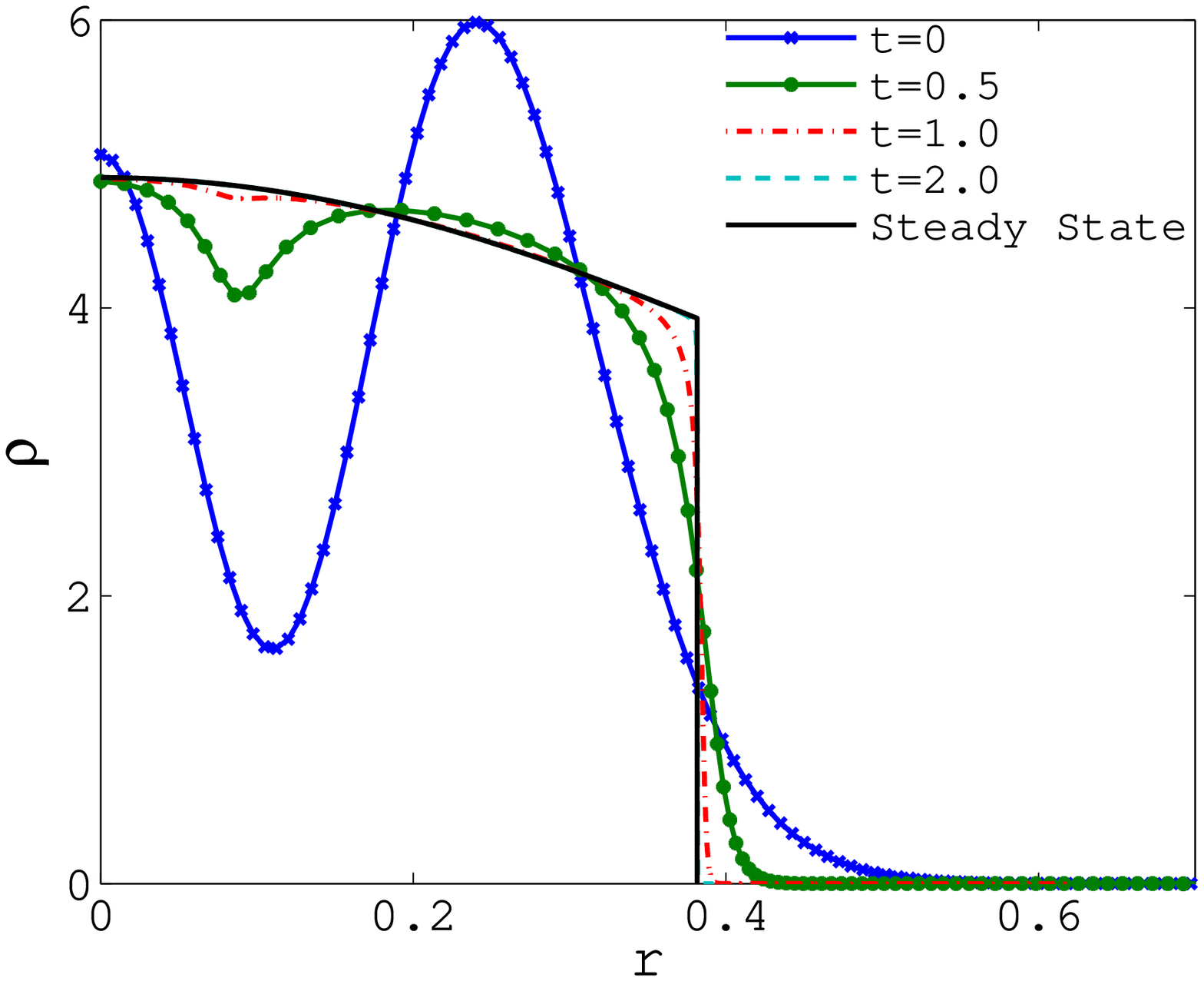} $~~~~$
\includegraphics[width=0.46\textwidth]{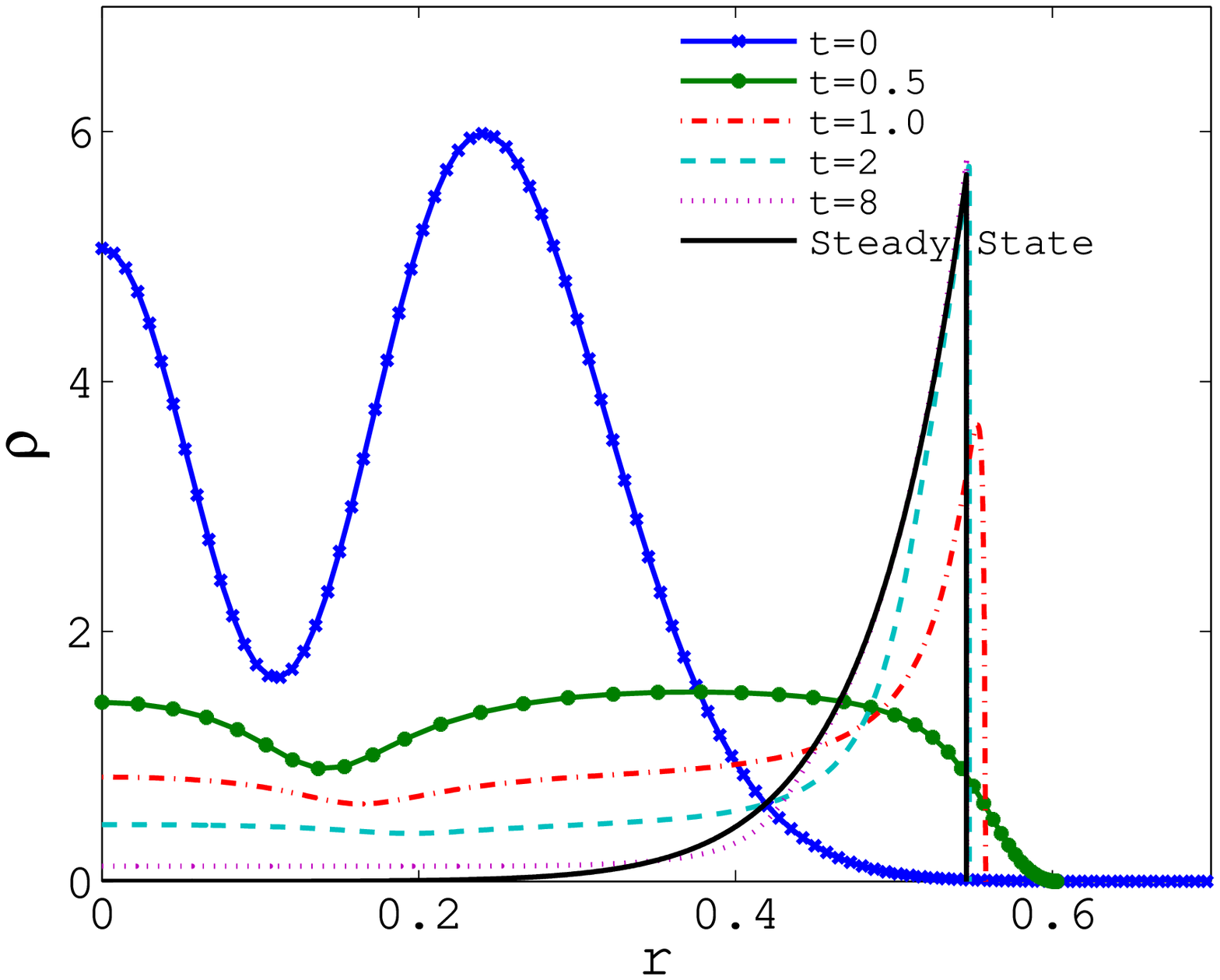}
 \end{center}
\[
\hspace{0.7cm} \textbf{(a)} \hspace{7.8cm} \textbf{(b)} 
\]
\caption{Time evolution of a radially symmetric solution to the aggregation 
model~\eqref{eq:aggre}-\eqref{eq:kernel} with (a) $q=1.5$ and (b) $q=20$, in three dimensions,
starting from initial data~\eqref{eq:initrho}. The solutions approach asymptotically the steady states 
studied in Section~\ref{sect:sstates} (solutions to~\eqref{eqn:int-eq}) and shown in 
Figure~\ref{fig:variousq}. Numerics with a variety of other initial conditions suggests that these
 equilibria are global attractors for the dynamics of~\eqref{eq:aggre}-\eqref{eq:kernel}.}
\label{fig:dynamic}
\end{figure}


\subsection{Asymptotic behaviour of equilibria as $q \to \infty$ and $q \searrow 2-n$} 
\label{subsect:asympt}

In general, there is no explicit formula for the radius of the support $R$ and the corresponding 
steady states $\bar{\rho}$ on $B(0,R)$, governed by~\eqref{eqn:int-eq}. However, when $q$ 
is large or close to the singular limit $2-n$, the asymptotic behaviours can be obtained by perturbation expansions. 
To facilitate the exposition, we only consider the case $R=1$ and solve the eigenvalue problem 
\begin{equation}
\label{eqn:evp-R1}
T_1\bar{\rho}_1(x)= \lambda\bar{\rho}_1(x), \qquad \text{ for } x\in  B(0,1),
\end{equation}
where 
\begin{equation}
\label{eqn:int-eqscaled}
 T_1\bar{\rho}_1(x) = (n+q-2)\int_{B(0,1)}|x-y|^{q-2}\bar{\rho}_1(y)dy,
\end{equation}
and $\bar{\rho}_1$ vanishes outside $B(0,1)$. Provided we know  the spectral radius $\lambda$ of $T_1$ and its 
corresponding eigenfunction $\bar{\rho}_1$, the actual steady state $\bar{\rho}$ in~\eqref{eqn:int-eq} can be recovered as 
$\bar{\rho}(r)=c\bar{\rho}_1(r/R)$,  where $R=\lambda^{-1/(n+q-2)}$ and $c$ is a constant
related to the total mass (see the proof of Theorem~\ref{thm:PF}, in particular equations~\eqref{eqn:rho-scale} 
and~\eqref{eqn:R-gen}).

In the rest of this section, the total mass for $\bar{\rho}_1$ is assumed to be one.
All asymptotic behaviours are investigated in one dimension first, 
to illustrate the essential techniques and features in a relatively simple setting, and then in higher dimensions.


\paragraph{Asymptotic limit when $q\to\infty$.} In this limit, the numerically computed eigenfunctions   are
observed to be concentrated near the boundary and the dominant contribution of the integral 
operator~\eqref{eqn:int-eqscaled} comes from $|y|\approx 1$. This motivates the following asymptotic
construction.

Start with the eigenvalue problem~\eqref{eqn:evp-R1} in one dimension ($n=1$):
\begin{equation}\label{eq:onedlarge}
 \lambda\bar{\rho}_1(x) = (q-1)\int_{-1}^1 |x-y|^{q-2}\bar{\rho}_1(y)dy.
\end{equation}
When $x>0$, the above integral on the right hand side is dominated at $y\approx-1$, that is
\begin{align}
 (q-1)\int_{-1}^1 |x-y|^{q-2}\bar{\rho}_1(y)dy 
 & \approx (q-1)\bar{\rho}_1(-1)\int_{-1}^1 |x-y|^{q-2} dy\cr
& = \bar{\rho}_1(-1) \left( (1+x)^{q-1} + (1-x)^{q-1}\right). \label{eqn:approx1}
\end{align}
For $x<0$ we find the same expression, considering the symmetry $\bar{\rho}_1(1) = \bar{\rho}(-1)$. Evaluate~\eqref{eq:onedlarge} 
at $x=-1$ and use~\eqref{eqn:approx1} to approximate its right-hand-side. We find
$\lambda \approx 2^{q-1}$ and 
\begin{equation}\label{eqn:apprho1d}
 \bar{\rho}_1(x) \approx \frac{q}{2^{q+1}}\left( (1+x)^{q-1} + (1-x)^{q-1}\right),
\end{equation}
where the coefficient $q/2^{q+1}$ is chosen such that the approximated $\bar{\rho}_1$ has total mass one. 

Approximations, as $q\to \infty$, to the actual steady states (with unit mass) and their support can be found 
from~\eqref{eqn:rho-scale} and~\eqref{eqn:R-gen}. Hence, $R = \lambda^{-1/(q-1)} \approx 0.5$. However, 
this is not an accurate approximation, as it can be observed from Figure~\ref{fig:1dlarge}(a), where we plot the 
radius $R$ against $q$, as obtained from numerics (dots) and the above approximation (dashed line), the latter being 
referred to as the {\em coarse} approximation. To obtain the numerical results we used the methods described in 
Section~\ref{subsect:numerics}.

The approximation of the eigenvalue $\lambda$ 
can be improved considerably if in the right-hand-side of~\eqref{eq:onedlarge} one uses the expression of $\bar{\rho}_1$ 
given by~\eqref{eqn:apprho1d}.  Near $y \approx -1$, use $\bar{\rho}_1(y)\approx q (1-y)^{q-1}/2^{q+1}$ to approximate 
the right-hand-side in~\eqref{eq:onedlarge}. Then, evaluate at $x=1$ to find
\[
 \lambda \bar{\rho}_1(1) \approx  \frac{q(q-1)}{2^{q+1}}\int_{-1}^1 |1-y|^{2q-3}dy.
\]
The integral in the right-hand-side can be computed exactly and also, from~\eqref{eqn:apprho1d},  
$\bar{\rho}_1(1) \approx q/4$. We derive the {\em refined} approximation $\lambda \approx 2^{q-2}$, with  the corresponding radius,
\begin{equation}
\label{eqn:R-refined}
R \approx 2^{-(q-2)/(q-1)}.
\end{equation}
The refined approximation, displayed as solid line in Figure~\ref{fig:1dlarge}(a) shows an excellent 
agreement with the numerical results (dots). Note that $R\to 0.5$ (the coarse approximation), as $q\to \infty$, 
but the convergence is slow. 

Formally, the eigenfunction~\eqref{eqn:apprho1d} can be regarded as 
$\bar{\rho}^{(1)}$ obtained from the power method iteration~\eqref{eqn:powermethod}, by starting with a constant initial guess $\bar{\rho}^{(0)}\equiv 1$. The coarse approximation $2^{q-1}$ of the eigenvalue is exactly $\lambda^{(1)}$, while the refined approximation $2^{q-2}$ is $\lambda^{(2)}$. Here the norm
 $\|\cdot\|$ is the function evaluation at $x=1$, i.e., $\|\bar{\rho}\|=\bar{\rho}(1)$. The same idea is
applied below in higher dimensions, even though the expressions are much more complicted.
This fact also illustrates the fast convergence of the iterative scheme~\eqref{eqn:powermethod} for large $q$.

We also compare the steady states, as obtained by numerics (see methods in Section~\ref{subsect:numerics}) and 
asymptotics (expression~\eqref{eqn:apprho1d}). Figure~\ref{fig:1dlarge}(b) shows an excellent agreement between 
the two solutions for $q=10$ and $q=20$.

\begin{figure}[htb]
 \begin{center}
\includegraphics[totalheight=0.28\textheight]{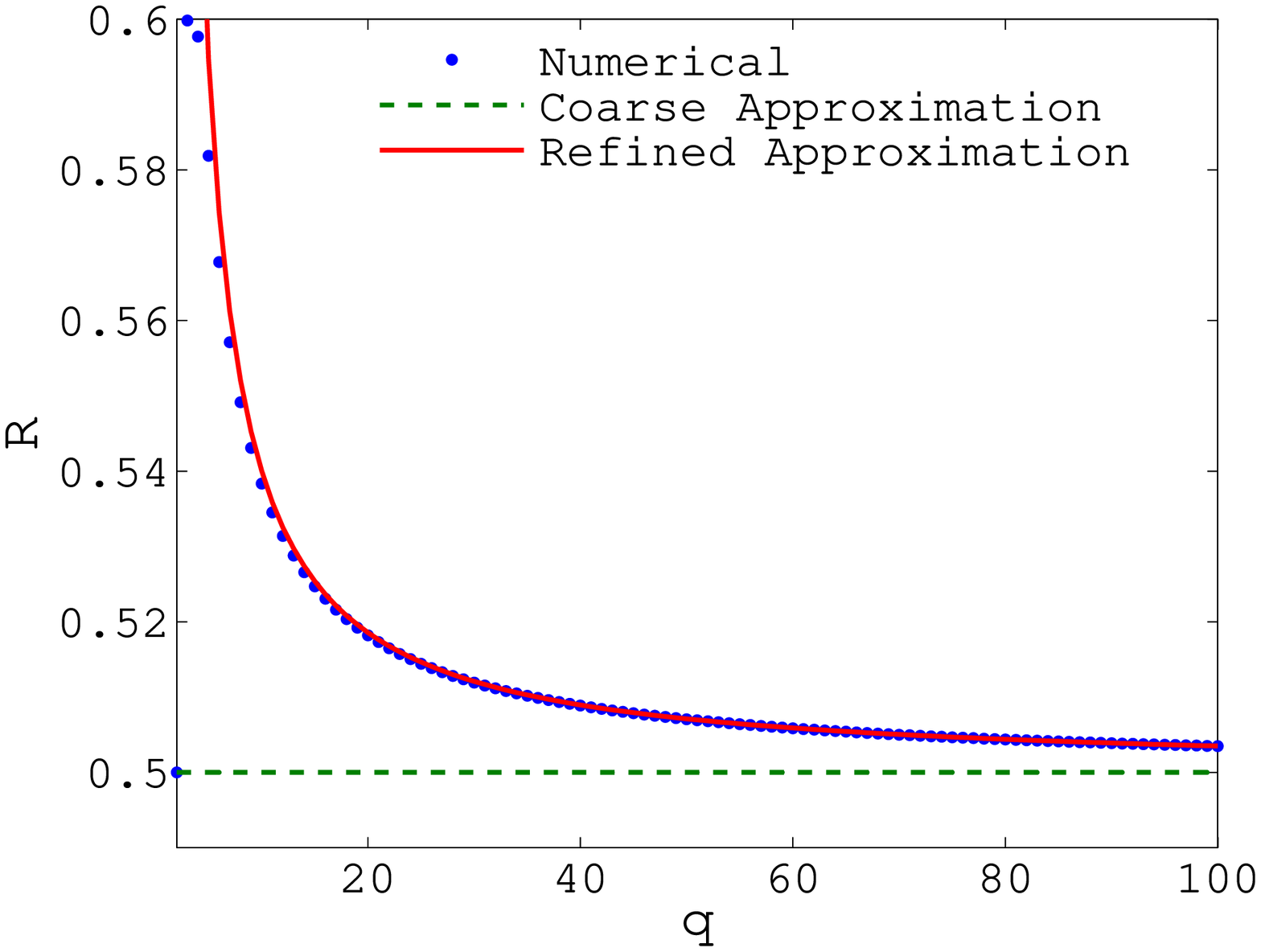}
\includegraphics[totalheight=0.28\textheight]{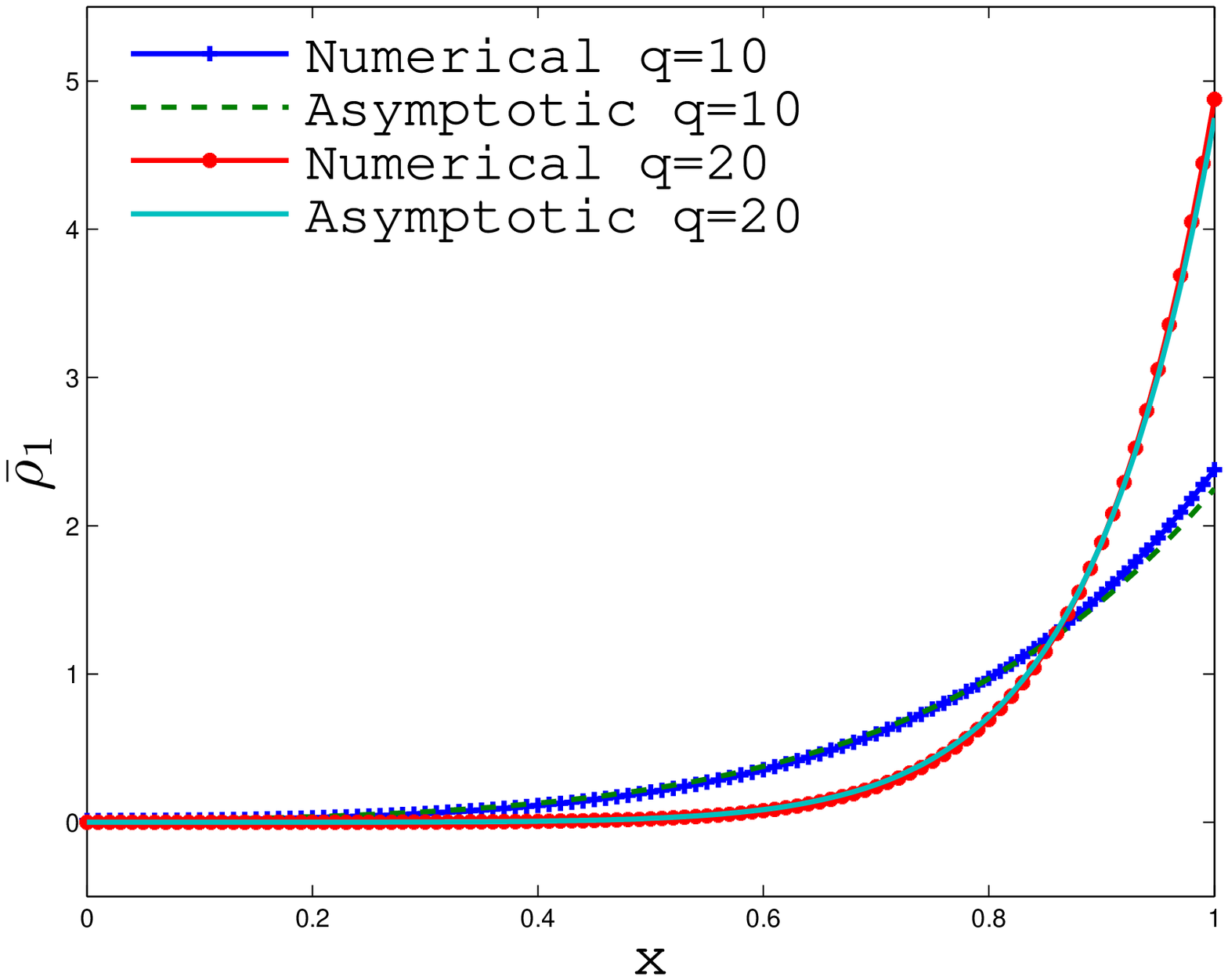}
\end{center}
\[
\hspace{0.8cm} \textbf{(a)} \hspace{7cm} \textbf{(b)} 
\]
\caption{Numerical and asymptotic solutions in dimension one, as $q\to \infty$. (a) Radius $R$ of the steady 
states~\eqref{eqn:int-eq}. The refined asymptotic approximation~\eqref{eqn:R-refined} (solid line) agrees 
extremely well with the numerical solution (dots). The coarse approximation $R=0.5$ (dashed line) captures 
the (slow) asymptotic limit, as $q \to \infty$. (b) Eigenfunctions $\bar{\rho}_1$ of $T_1$ computed by 
numerical (Section~\ref{subsect:numerics}) and asymptotic (equation~\eqref{eqn:apprho1d}) methods. The agreement 
between the two sets of results is excellent.}
\label{fig:1dlarge}
\end{figure}

In higher dimensions,  the eigenvalue problem can be approximated in a similar manner. Based on numerical 
observations, we assume that most contribution in the integral from~\eqref{eqn:int-eqscaled} comes from the boundary. Hence, 
the eigenvalue problem for $T_1$ is approximated by 
\begin{equation}\label{eq:largqhigherd}
    \lambda \bar{\rho}_1(x) \approx (n+q-2)\bar{\rho}_1(1)\int_{B(0,1)}|x-y|^{q-2}dy.
\end{equation}
In general, the integral on the right-hand-side can not be integrated explicitly in the natural radial coordinates.
However, when the origin is shifted to $x$, we have
\begin{align}
 \int_{B(0,1)}|x-y|^{q-2}dy &= \int_{B(x,1)} |y|^{q-2}dy \nonumber \\
&= \frac{n \omega_n}
{\int_0^\pi \sin^{n-2}\theta d\theta}\int_0^\pi \sin^{n-2}\theta
\int_{r^2+r'^2-2rr'\cos\theta\leq 1} (r')^{n+q-3}dr'd\theta \nonumber \\
&=\frac{n \omega_n}{(n+q-2)\int_0^\pi \sin^{n-2}\theta d\theta} \int_0^\pi \sin^{n-2}
\theta(\sqrt{1-r^2\sin^2\theta}-r\cos\theta)^{n+q-2}d\theta.
\label{eqn:shift}
\end{align}

The approximation to the eigenvalue $\lambda$ is obtained by evaluating 
expression~\eqref{eq:largqhigherd} at $r=1$, which gives
\begin{align}
    \lambda &\approx  \frac{n \omega_n}{\int_0^\pi \sin^{n-2}\theta d\theta}\int_0^\pi \sin^{n-2}\theta 
(|\cos\theta|-\cos\theta)^{n+q-2}d\theta \nonumber \\
&= (n-1) \omega_{n-1} \, 2^{n+q-3}\mbox{Beta}\left(\frac{n-1}{2},\frac{n+q-1}{2}\right) \cr
&= (n-1)\omega_{n-1}2^{n+q-3}\Gamma(\frac{n-1}{2})\Gamma(\frac{n+q-1}{2})/\Gamma(n-1+\frac{q}{2}),
\label{eqn:lambda-3d}
\end{align}
where $\mbox{Beta}$ and $\Gamma$ are the beta and Gamma functions, respectively. 

The radius of the support $R = \lambda^{-1/(n+q-2)}$ can then be approximated as $q\to \infty$ using  Stirling's approximation,
\[
 \Gamma(z+1) \sim \sqrt{2\pi}z^{z+1/2}e^{-z}.
\]
We refer to the outcome of this procedure as the {\em coarse} approximation and we plot the result  in Figure \ref{fig:highdlargq}(a) (dashed line). Note the low accuracy of this approximation, when compared to the numerical calculation (dots). The coarse approximation approaches however,  as $q\to \infty$, the correct value $0.5$, but the convergence is extremely slow.

The eigenfunction $\bar{\rho}_1$ is approximated from~\eqref{eq:largqhigherd}:
\begin{equation}\label{eq:largqhdrho}
    \bar{\rho}_1(r) \approx C\int_0^\pi \sin^{n-2}\theta(\sqrt{1-r^2\sin^2\theta}-r\cos\theta)^{n+q-2}d\theta,
\end{equation}
where $C$ is a normalization constant determined by the total mass. 

Similar to the 1D case, using the explicit approximation~\eqref{eq:largqhdrho} of the eigenfunction, one can  improve considerably the approximation of the corresponding eigenvalue $\lambda$ (and hence of $R$), by
substituting~\eqref{eq:largqhdrho} into~\eqref{eq:largqhigherd} and evaluating at $r=1$.  The outcome of this improved procedure is referred to as the {\em refined} approximation and is plotted (solid line) in Figure~\ref{fig:highdlargq}(a).  The agreement of the refined approximation with the numerical results is now excellent. There is also a very good agreement between the eigenfunctions computed numerically (see Section~\ref{subsect:numerics}) and their asymptotic approximation provided by \eqref{eq:largqhdrho} --- see Figure \ref{fig:highdlargq}(b).

{\em Remark.} Since the integral expression~\eqref{eq:largqhdrho} for the approximation of $\bar{\rho}_1$ is concentrated at 
$\theta\approx \pi$, we can use Laplace's method to find the leading order when $r$ is away from zero. Since 
$\bar{\rho}_1(1)=C2^{n+q-3}\mbox{Beta}\left(\frac{n-1}{2},\frac{n+q-1}{2}\right)$, and for $\theta \approx \pi$,
\[
 \sqrt{1-r^2\sin^2\theta}-r\cos\theta = (1+r)\left[1-r(\pi-\theta)^2/2+\cdots\right],
\]
the eigenfunction $\bar{\rho}_1$ can be further approximated away from the origin as
\begin{align*}
 \bar{\rho}_1(r) &\approx C(1+r)^{n+q-2}\int_0^\pi \sin^{n-2}\theta (1-r(\pi-\theta)^2)^{n+q-2}d\theta \cr
&= C(1+r)^{n+q-2} \int_0^\infty \theta^{n-2}e^{-(n+q-2)r\theta^2}d\theta\cr
&=Cr^{-(n-1)/2}(1+r)^{n+q-2}.
\end{align*}
However, this expression breaks down for $r$ near the origin and approximation \eqref{eq:largqhdrho} is used instead in the plot from Figure~\ref{fig:highdlargq}(b).

To conclude, the asymptotic study as $q \to \infty$ shows that the radii of support  of the equilibria \eqref{eqn:int-eq} converge (slowly) to a fixed value $R=0.5$, while the density 
concentrates to a $\delta$-sphere of radius 0.5. We also point out  that this asymptotic behaviour applies to all dimensions $n$.

\begin{figure}[htb]
 \begin{center}
  \includegraphics[totalheight=0.27\textheight]{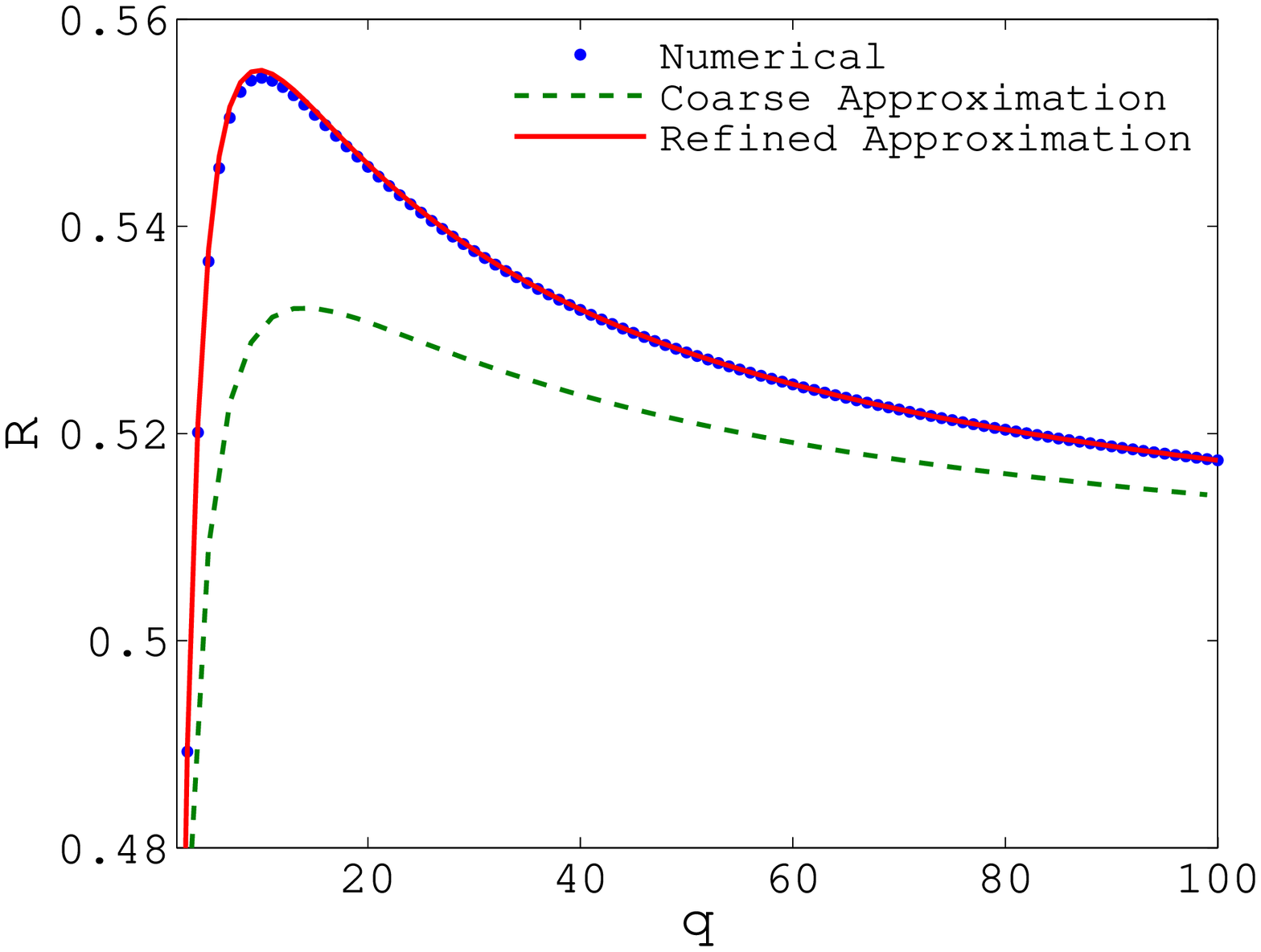}
\includegraphics[totalheight=0.27\textheight]{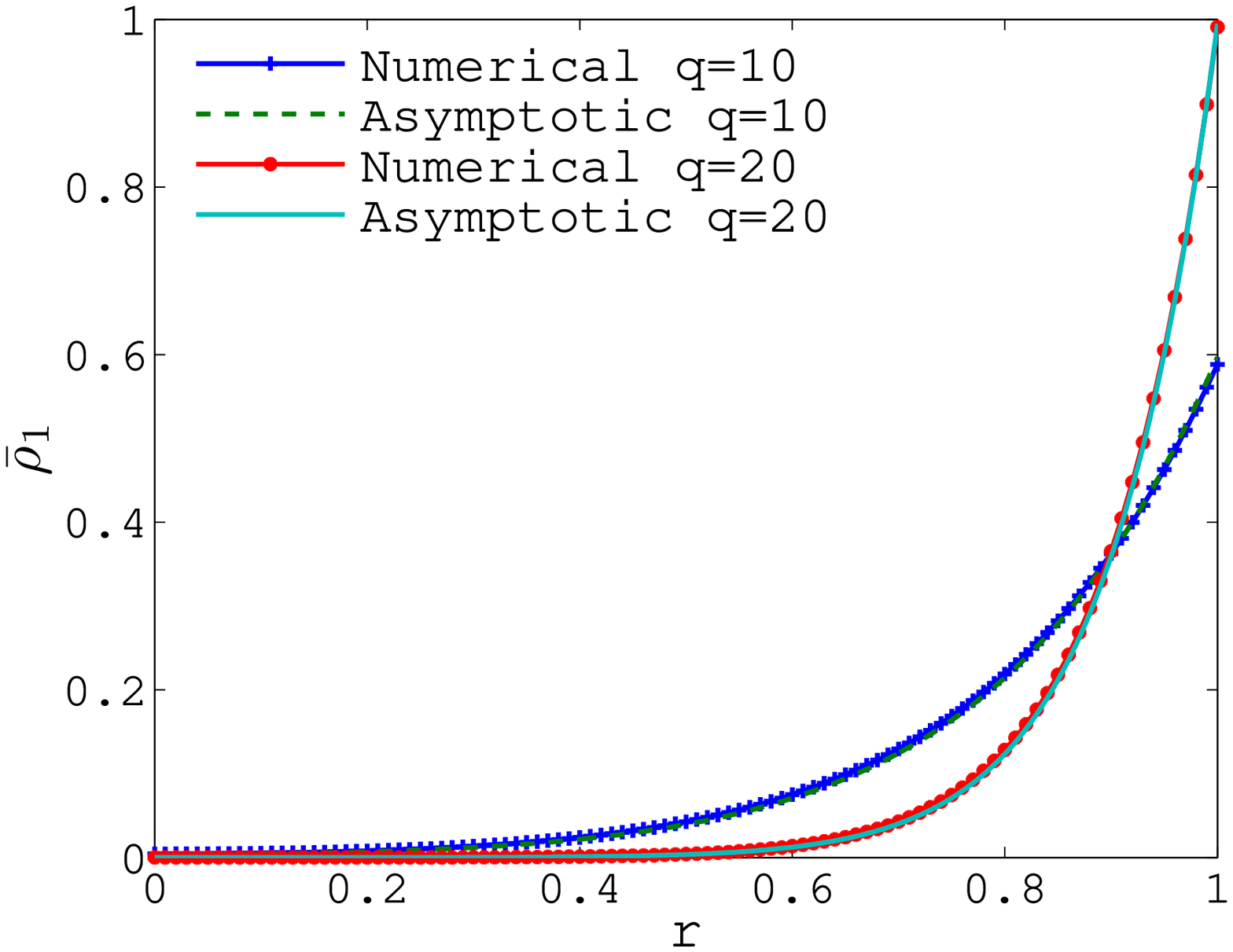}
\end{center}
\[
\hspace{0.8cm} \textbf{(a)} \hspace{7.5cm} \textbf{(b)} 
\]
\caption{Numerical and asymptotic solutions in dimension three, as $q\to \infty$. 
(a) Radius $R$ of the steady states~\eqref{eqn:int-eq}. The refined asymptotic approximation (solid line) and the numerical solution (dots) agree extremely well. The dashed line represent the coarse approximation obtained from  \eqref{eqn:lambda-3d} and Stirling's formula.  The radii $R$ approach (very slowly)  
the constant $0.5$, as $q \to \infty$. (b) Eigenfunctions $\bar{\rho}_1$ of $T_1$ computed by numerical and asymptotic  (approximation  \eqref{eq:largqhdrho}) methods. There is very good agreement between the two solutions. }
\label{fig:highdlargq}
\end{figure}


\paragraph{Asymptotic limit when $q\to 2-n$.} We write  
\[ 
q= 2-n +\epsilon,
\]
 and perform an asymptotic study in the small $\epsilon$ regime. We start with solutions in one dimension first, where 
the eigenvalue problem~\eqref{eqn:evp-R1} reads
\begin{equation}\label{eq:sing1d}
 \lambda_\epsilon \bar{\rho}_1^\epsilon(x) = 
\epsilon\int_{-1}^1 |x-y|^{\epsilon-1}\bar{\rho}_1^\epsilon(y)dy,\qquad 
\epsilon = q-1.
\end{equation}
Here the subscript and superscript $\epsilon$ is used to emphasize the dependence of the eigenvalue $\lambda$ and 
the corresponding eigenfunction $\bar{\rho}_1$ on $\epsilon$. The asymptotic expansion suggested from numerical simulation is
\begin{subequations}
\label{eqn:expansion}
\begin{equation}\label{eq:lambda1eps}
 \lambda_\epsilon = \lambda_0+\lambda_1\epsilon+\lambda_2\epsilon^2+\cdots,
\end{equation}
\begin{equation}\label{eq:rho1eps}
 \bar{\rho}_1^\epsilon(x) = \bar{\rho}^{(0)}(x)+\epsilon \bar{\rho}^{(1)}(x)+\epsilon^2\bar{\rho}^{(2)}(x)+\cdots.
\end{equation}
\end{subequations}

The kernel $|x-y|^{\epsilon-1}$ is not integrable in the limit when $\epsilon\to 0$, and thus 
we can not substitute the formal expansions into~\eqref{eq:sing1d} directly.
Hence, it is not possible to carry out a straightforward expansion. However, using the fact that
\begin{equation*}
 \epsilon\int_{-1}^1 |x-y|^{\epsilon-1} dy = (1+x)^{\epsilon}+(1-x)^{\epsilon},
\end{equation*}
the governing equation~\eqref{eq:sing1d} can be written as
\begin{equation}
\label{eqn:mod1d}
 \left[\lambda_\epsilon-(1-x)^\epsilon-(1+x)^\epsilon\right]\bar{\rho}^\epsilon_1(x)
=\epsilon\int_{-1}^1 |x-y|^{\epsilon-1}\left[\bar{\rho}^\epsilon_1(y)-\bar{\rho}^\epsilon_1(x)\right]dy.
\end{equation}
Now using the asymptotic expansions~\eqref{eq:lambda1eps} and~\eqref{eq:rho1eps}, we obtain
\begin{subequations}
\begin{align}
 O(1):&\qquad \lambda_0-2 = 0 \label{eq:1d0}\\
 O(\epsilon):&\qquad \int_{-1}^1 |y-x|^{-1}(\bar{\rho}^{(0)}(y)-\bar{\rho}^{(0)}(x))dy-
(\lambda_1-\ln(1-x^2))\bar{\rho}^{(0)}(x))=0\label{eq:1d1}\\
O(\epsilon^2):&\qquad \int_{-1}^1 |y-x|^{-1}(\bar{\rho}^{(1)}(y))-\bar{\rho}^{(1)}(x)))dy
-(\lambda_1-\ln(1-x^2))\bar{\rho}^{(1)}(x))\cr
& = \left(
\lambda_2-\frac{\ln^2(1-x)+\ln^2(1+x)}{2}\right)\bar{\rho}^{(0)}(x))
-\int_{-1}^1 \frac{\ln|x-y|}{|y-x|}(\bar{\rho}^{(0)}(y))-\bar{\rho}^{(0)}(x)))dy.
\label{eq:1d2}
\end{align}
\end{subequations}
The first equation~\eqref{eq:1d0} yields $\lambda_0=2$ and the second equation~\eqref{eq:1d1} is 
an eigenvalue problem for the limiting profile $\bar{\rho}^{(0)}$. 
Integrating the equation~\eqref{eq:1d1} with respect to $x$, we can get an
alternative expression for the eigenvalue $\lambda_1$ as
\begin{equation}\label{eq:lambda1exp}
 \lambda_1 = \frac{\int_{-1}^1 \bar{\rho}^{(0)}(x)\ln(1-x^2)dx}{\int_{-1}^1 \bar{\rho}^{(0)}(x)dx}.
\end{equation}
Even though we do not expect any explicit solutions to~\eqref{eq:1d1}, we can discretize and solve
it using inverse iteration~\cite{GolubVanLoan}. The condition $\bar{\rho}^{(0)}(\pm 1)=0$,
motivated from the numerical calculation of the steady states, is used to get rid of
 the singularity on the boundary. The initial data for the inverse iteration can be taken as 
the steady state $\bar{\rho}_1^\epsilon$ calculated numerically from~\eqref{eq:sing1d} 
with $\epsilon$ close to zero (and initial guess of the eigenvalue from~\eqref{eq:lambda1exp}  
with $\bar{\rho}^{(0)}$ replaced by $\bar{\rho}_1^\epsilon$).  This inverse power iteration normally converges in 
just a few steps.

The solution $\bar{\rho}^{(1)}$ of~\eqref{eq:1d2} is more challenging. Since this first order correction 
$\bar{\rho}^{(1)}$ does not provide much insight into the problem, we focus instead on the second order 
correction $\lambda_2$ of the eigenvalue $\lambda$. 
By the solvability condition for~\eqref{eq:1d2}, the right hand side of~\eqref{eq:1d2} is orthogonal to 
$\bar{\rho}^{(0)}$, giving 
\begin{align}
 \lambda_2\int_{-1}^1 \left[\bar{\rho}^{(0)}(x)\right]^2dx &=
 \int_{-1}^1 \frac{\ln|y-x|}{|y-x|}(\bar{\rho}^{(0)}(y)-\bar{\rho}^{(0)}(x))\bar{\rho}^{(0)}(x)dydx\cr
&\quad +\int_{-1}^1 \frac{\ln^2(1-x)+\ln^2(1+x)}{2}\left[\bar{\rho}^{(0)}(x)\right]^2dx.
\label{eq:lambda2exp}
\end{align}

Figure~\ref{fig:ldsmall}(a) shows the normalized steady states $\bar{\rho}_1^\epsilon$ for $\epsilon=1,0.5,0.2$ 
(corresponding to $q=2,1.5,1.2$), as computed numerically from the eigenvalue problem~\eqref{eq:sing1d} using the 
methods described in Section~\ref{subsect:numerics}. In the same figure we also plot (plain solid line) the leading 
order term $\bar{\rho}^{(0)}$ of the expansion~\eqref{eq:rho1eps}, obtained by solving~\eqref{eq:1d1} with the 
inverse iteration method. The plot confirms that the equilibria $\bar{\rho}_1^\epsilon$ approach the limiting 
profile $\bar{\rho}^{(0)}$ as $\epsilon \to 0$.

In Figure~\ref{fig:ldsmall}(b) we plot (dots) the eigenvalues $\lambda_\epsilon$  for various 
values of $\epsilon \in (0,1]$ (corresponding to $q \in (1,2]$), computed directly from~\eqref{eq:sing1d} using 
the power method (Section~\ref{subsect:numerics}). The dashed line represents the asymptotic approximation 
of $\lambda^\epsilon$ from~\eqref{eq:lambda1eps} that includes second order corrections ($\lambda_1$ is computed 
from~\eqref{eq:lambda1exp} and $\lambda_2$ from~\eqref{eq:lambda2exp}. The agreement between the two sets of results 
for small $\epsilon$'s is excellent. 

\begin{figure}[htb]
    \begin{center}
        \includegraphics[totalheight=0.27\textheight]{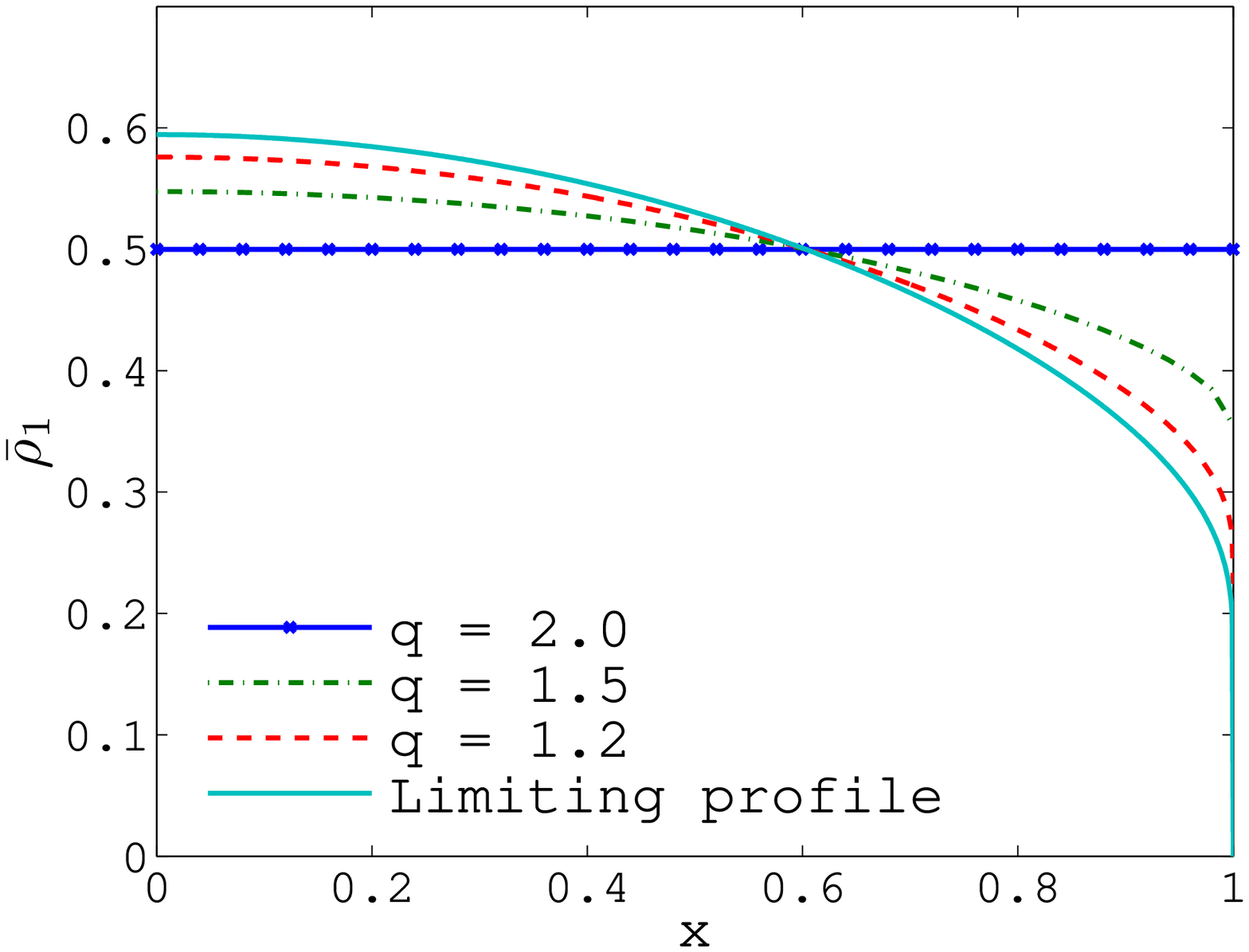}
        \includegraphics[totalheight=0.27\textheight]{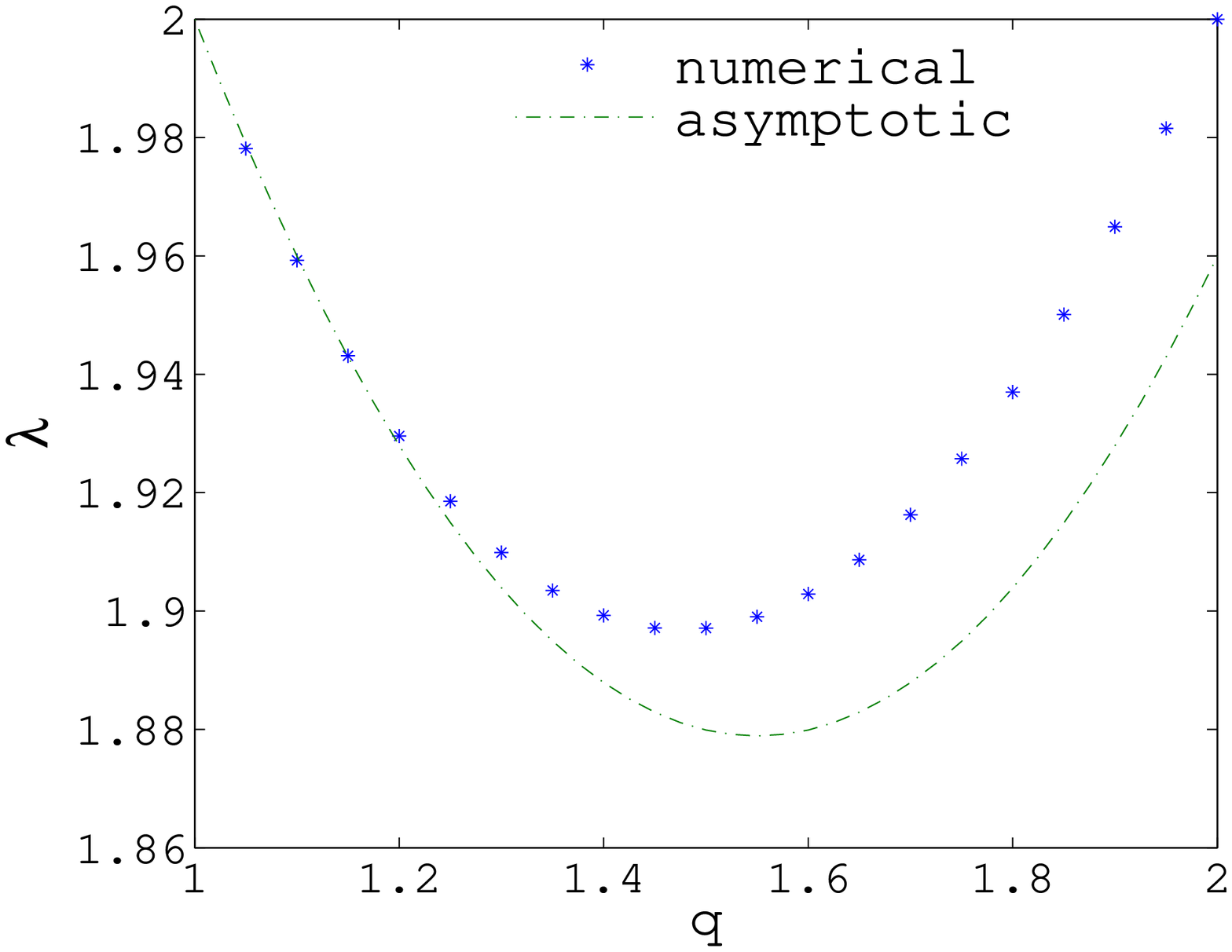}
    \end{center}
\[
\hspace{0.8cm} \textbf{(a)} \hspace{7.5cm} \textbf{(b)} 
\]
    \caption{ (a) Normalized steady states $\bar{\rho}_1^\epsilon$ for $\epsilon=1,0.5,0.2$ ($q=2,1.5,1.2$) 
computed directly from~\eqref{eq:sing1d} using the numerical methods from Section~\ref{subsect:numerics}. The plain 
solid line represents the leading order term $\bar{\rho}^{(0)}$ of the expansion~\eqref{eq:rho1eps}. Equilibria 
$\bar{\rho}_1^\epsilon$ approach the limiting profile $\bar{\rho}^{(0)}$ as $\epsilon \to 0$. 
(b) Eigenvalues $\lambda_\epsilon$  for various values of $\epsilon \in (0,1]$ ($q \in (1,2]$), computed 
directly from~\eqref{eq:sing1d} (dots). The dashed line is the asymptotic approximation of $\lambda^\epsilon$ 
from~\eqref{eq:lambda1eps} valid at order $O(\epsilon^2)$. Note the excellent agreement between numerics and asymptotics 
for small $\epsilon$ ($q$ close to $1$).
}
\label{fig:ldsmall}
\end{figure}

At a closer inspection, it becomes clear that the expansion~\eqref{eq:rho1eps} may be non-uniform
near the boundary $x=1$. In Figures~\ref{fig:ldsmallbv}(a)-(b)  we plot  $\bar{\rho}_1^{\epsilon}$ evaluated 
at $x=0$ and $x=1$, respectively,  for different values of $\epsilon$. The eigenfunctions $\bar{\rho}_1^{\epsilon}$  are 
computed directly from~\eqref{eq:sing1d} --- see also Figure~\ref{fig:ldsmall}(a). At the origin 
(Figure~\ref{fig:ldsmallbv}(a)) we find a linear dependence on $\epsilon$, as the approximation $\bar{\rho}^\epsilon_1(0)=
\bar{\rho}^{(0)}(0)+\epsilon\bar{\rho}^{(1)}(0)+\cdots$ is uniform. However, at $x=1$, we find 
$\bar{\rho}^\epsilon_1(1)\sim \sqrt{0.253\epsilon+O(\epsilon^2)}$, hence the expansion~\eqref{eq:rho1eps} is non-uniform 
near the boundary (Figure \ref{fig:ldsmallbv}(b)). To find a valid asymptotic expansion near $x=1$ one has to  introduce  a boundary layer and perhaps use the method of matched asymptotics to relate the inner and outer expansions. We do not pursue this  direction here.

\begin{figure}[htb]
    \begin{center}
        \includegraphics[totalheight=0.27\textheight]{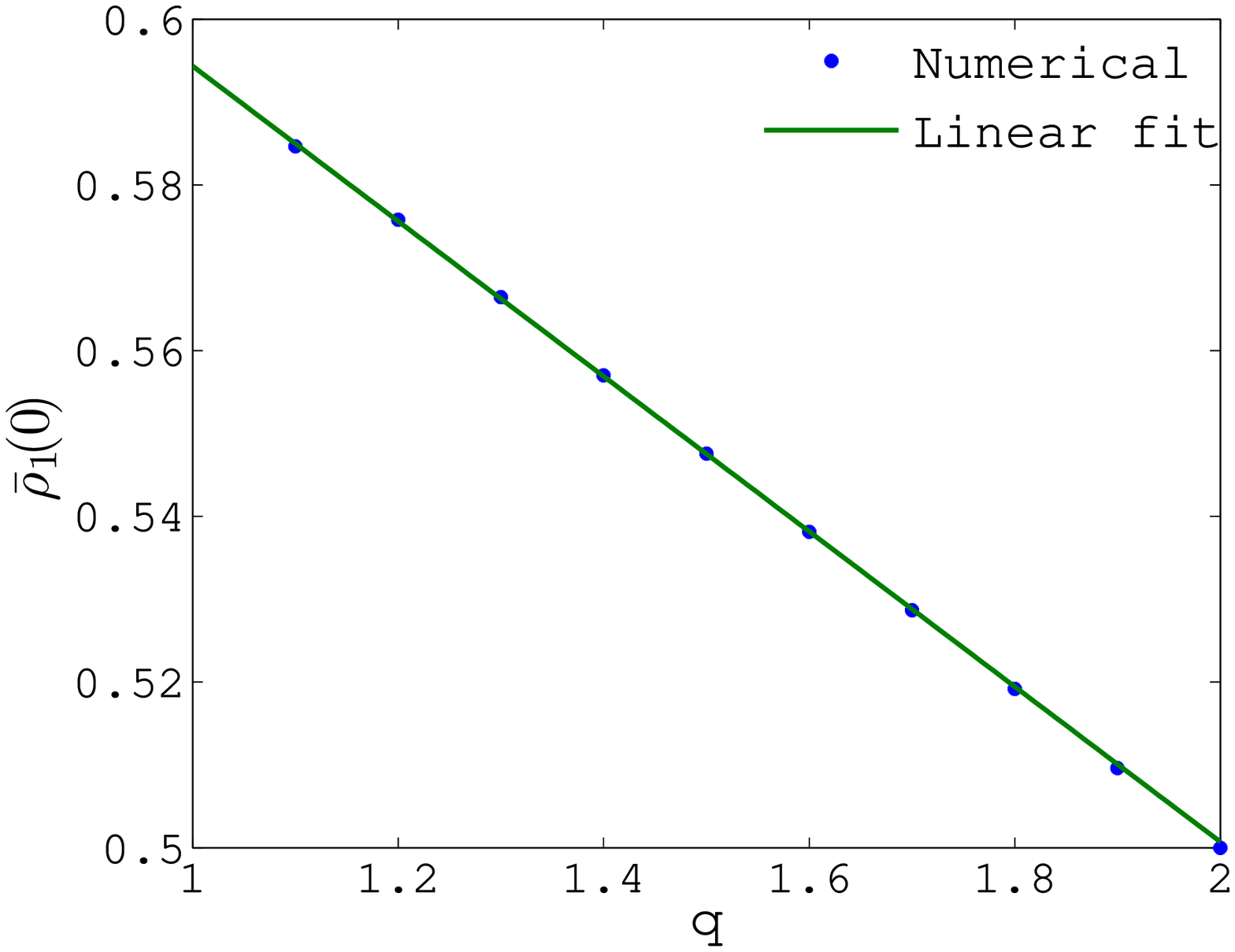}
        \includegraphics[totalheight=0.27\textheight]{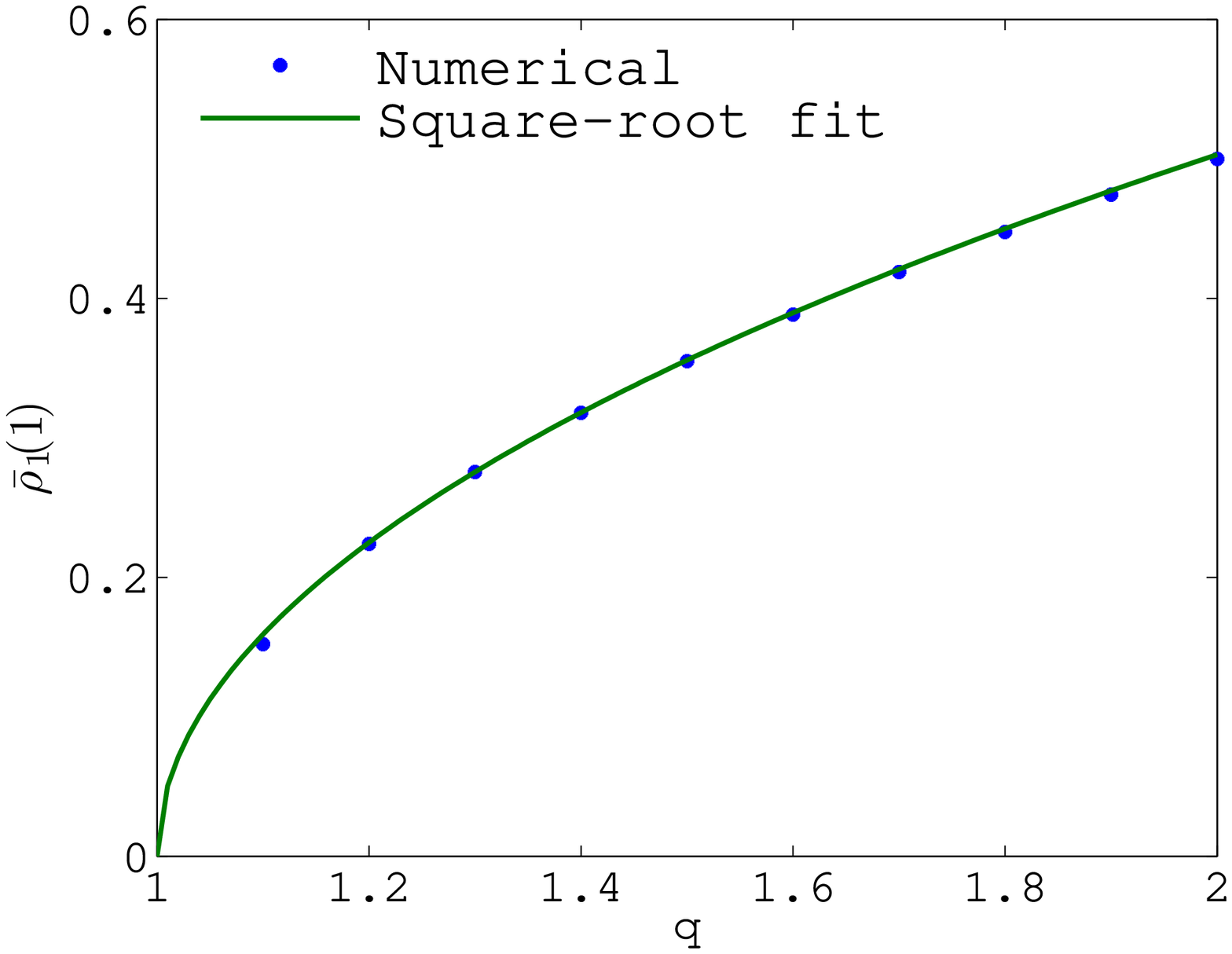}
    \end{center}
\[
\hspace{0.8cm} \textbf{(a)} \hspace{7.5cm} \textbf{(b)} 
\]
\caption{(a) Validation of the asymptotic expansion~\eqref{eq:rho1eps} away from the boundary. At $x=0$, 
$\bar{\rho}^\epsilon_1(0)$ has a linear dependence on $\epsilon$ of the form $0.5944-0.0936\epsilon$, 
as the expansion~\eqref{eq:rho1eps} is 
uniform near the origin. (b) At $x=1$,   $\bar{\rho}^\epsilon_1(1)$ has a square-root dependence on $\epsilon$ of the form
$\sqrt{0.2530\epsilon}$. 
There is a boundary layer near $x=1$ where the expansion~\eqref{eq:rho1eps} is no longer valid. }
\label{fig:ldsmallbv}
\end{figure}

The situation in higher dimensions is similar. We write the eigenvalue problem~\eqref{eqn:evp-R1} as
\begin{equation}\label{eq:highdsing}
 (\lambda_\epsilon - k_\epsilon(r))\bar{\rho}^\epsilon_1(r) = 
\epsilon\int_{B(0,1)}|x-y|^{\epsilon-n}(\bar{\rho}^\epsilon_1(y) -\bar{\rho}^{\epsilon}_1(x))dy,\qquad r=|x|,
\end{equation}
where we used again subscripts and superscripts to emphasize the dependence on $\epsilon=q+n-2$, and 
the auxiliary function $k_\epsilon$ is defined by
\begin{equation*}
 k_\epsilon(|x|) = \epsilon\int_{B(0,1)}|x-y|^{\epsilon-n}dy.
\end{equation*}
Equation~\eqref{eq:highdsing} is the analogue of~\eqref{eqn:mod1d} from 1D.

The auxiliary function $k_\epsilon(r)$ can be simplified by shifting the origin (see the calculation leading 
to~\eqref{eqn:shift}), with the result:
\begin{align*}
 k_\epsilon(r) &= \frac{n \omega_n}{\int_0^\pi \sin^{n-2}\theta d\theta}  \int_0^\pi \sin^{n-2}\theta 
(\sqrt{1-r^2\sin^2\theta}-r\cos\theta)^{\epsilon} d\theta \cr
&=  \frac{n \omega_n}{\int_0^\pi \sin^{n-2}\theta d\theta} \int_0^\pi \sin^{n-2}\theta 
\Bigl( 1+\epsilon \ln (\sqrt{1-r^2\sin^2\theta}-r\cos\theta) \cr
&\qquad +\frac{\epsilon^2}{2} \ln^2 (\sqrt{1-r^2\sin^2\theta}-r\cos\theta)+O(\epsilon^3) \Bigr) d\theta, \cr
\end{align*}
where we also used a Taylor expansion in $\epsilon$ to expand  $(\sqrt{1-r^2\sin^2\theta}-r\cos\theta)^{\epsilon}$. 
The $O(\epsilon)$ term can be simplified using the the following calculation:
\begin{align*}
&\quad \int_0^\pi \sin^{n-2}\theta\ln(\sqrt{1-r^2\sin^2\theta}-r\cos\theta)d\theta \cr
&= \int_0^{\pi/2} \sin^{n-2}\theta \left[\ln(\sqrt{1-r^2\sin^2\theta}-r\cos\theta)+
\ln(\sqrt{1-r^2\sin^2\theta}+r\cos\theta)\right]d\theta \cr
&=\ln(1-r^2)\int_0^{\pi/2}\sin^{n-2}\theta d\theta.
\end{align*}
Hence,
\begin{equation}
\label{eqn:k-eps}
 k_\epsilon(r) = n \omega_n  + \frac{n\omega_n}{2} \ln(1-r^2) \epsilon +k^{(2)}(r)\epsilon^2 +O(\epsilon^3),
\end{equation}
where 
\[
k^{(2)}(r) =  \frac{n\omega_n}{2\int_0^\pi \sin^{n-2}\theta d\theta} 
\int_0^\pi \sin^{n-2}\theta \ln^2(\sqrt{1-r^2\sin^2\theta}-r\cos\theta)d\theta.
\]

Since the integrand in~\eqref{eq:highdsing} is integrable provided $\bar{\rho}_1^\epsilon$ is H\"{o}lder continuous, 
by introducing the formal asymptotic expansion~\eqref{eqn:expansion} in~\eqref{eq:highdsing} and using~\eqref{eqn:k-eps}, one finds
\begin{subequations}
\begin{align}
 O(1): &\qquad \lambda_0 - n\omega_n \label{eq:nd0} =0 \\
O(\epsilon):& \qquad \int_{B(0,1)} |y-x|^{-n}(\bar{\rho}^{(0)}(x)-\bar{\rho}^{(0)}(y))dy -
(\lambda_1-\frac{1}{2} n \omega_n \ln(1-|x|^2))\bar{\rho}^{(0)}(x) = 0 \label{eq:nd1} \\
O(\epsilon^2): & \qquad \int_{B(0,1)}|y-x|^{-n}(\bar{\rho}^{(1)}(y)-\bar{\rho}^{(1)}(x))dy
-\left(\lambda_1- \frac{1}{2} n \omega_n \ln(1-|x|^2)\right)\bar{\rho}^{(1)}(x) =\cr
&\quad\qquad \left(\lambda_2-k^{(2)}(|x|)\right)\bar{\rho}^{(0)}(x) -
\int_{B(0,1)}|y-x|^{-n}\ln|y-x| (\bar{\rho}^{(0)}(y)-\bar{\rho}^{(0)}(x))dy. \label{eq:nd2}
\end{align}
\end{subequations}

Therefore, the leading order of the eigenvalue is $\lambda_0 = n\omega_n$ and the limiting steady state $\bar{\rho}^{(0)}$
and the first order correction $\lambda_1$ can be solved by inverse iteration. The second order correction
$\lambda_2$ can be obtained from the solvability condition:
\begin{align}
&\qquad \lambda_2 \int_{B(0,1)} \left[\bar{\rho}^{(0)}(x)\right]^2 dx
-\int_{B(0,1)} k^{(2)}(|x|)\left[\bar{\rho}^{(0)}(x)\right]^2  dx\cr
&=\int_{B(0,1)}\int_{B(0,1)}|y-x|^{-n}\ln|y-x| 
\left[\bar{\rho}^{(0)}(y)-\bar{\rho}^{(0)}(x))\right]\bar{\rho}^{(0)}(x) dy dx
\end{align}

In Figure~\ref{fig:higherdsmall} we test  the asymptotic results in three dimensions. The qualitative behaviour of 
the steady states $\bar{\rho}^\epsilon_1$ and their corresponding  eigenvalues $\lambda_\epsilon$ is similar 
to what has been observed in one dimension (Figure~\ref{fig:ldsmall}).

Figure~\ref{fig:higherdsmall}(a) shows the normalized steady states $\bar{\rho}_1^\epsilon$ for $\epsilon = 2,1,0.5,0.2$ 
(corresponding to $q = 1,0,-0.5,-0.8$) obtained by solving numerically~\eqref{eqn:evp-R1}, along with the limiting 
profile $\bar{\rho}^{(0)}$ (plain solid line) found by inverse iteration from~\eqref{eq:nd1}. The asymptotic results 
are confirmed, as equilibria $\bar{\rho}_1^\epsilon$ approach the limiting profile $\bar{\rho}^{(0)}$ for $\epsilon \to 0$.
In Figure~\ref{fig:higherdsmall}(b) we test the asymptotic expansion~\eqref{eq:lambda1eps} at order $O(\epsilon^2)$ 
(dashed line) against the numerical solution of~\eqref{eqn:evp-R1} (dots). The agreement is excellent for 
small $\epsilon$ or equivalently, for $q$ close to the critical value $-1$.

\begin{figure}[htb]
    \begin{center}
        \includegraphics[totalheight=0.26\textheight]{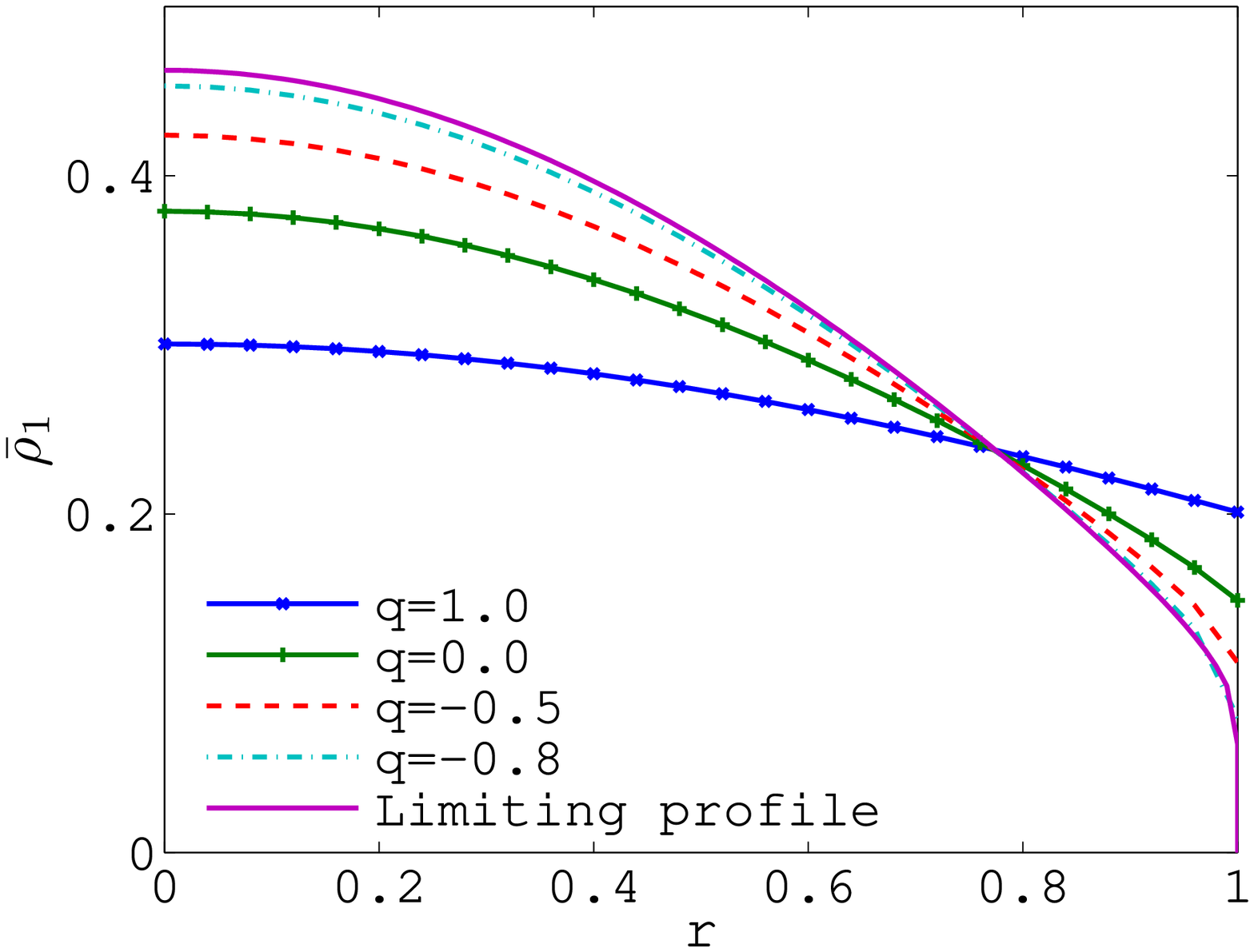}
        \includegraphics[totalheight=0.26\textheight]{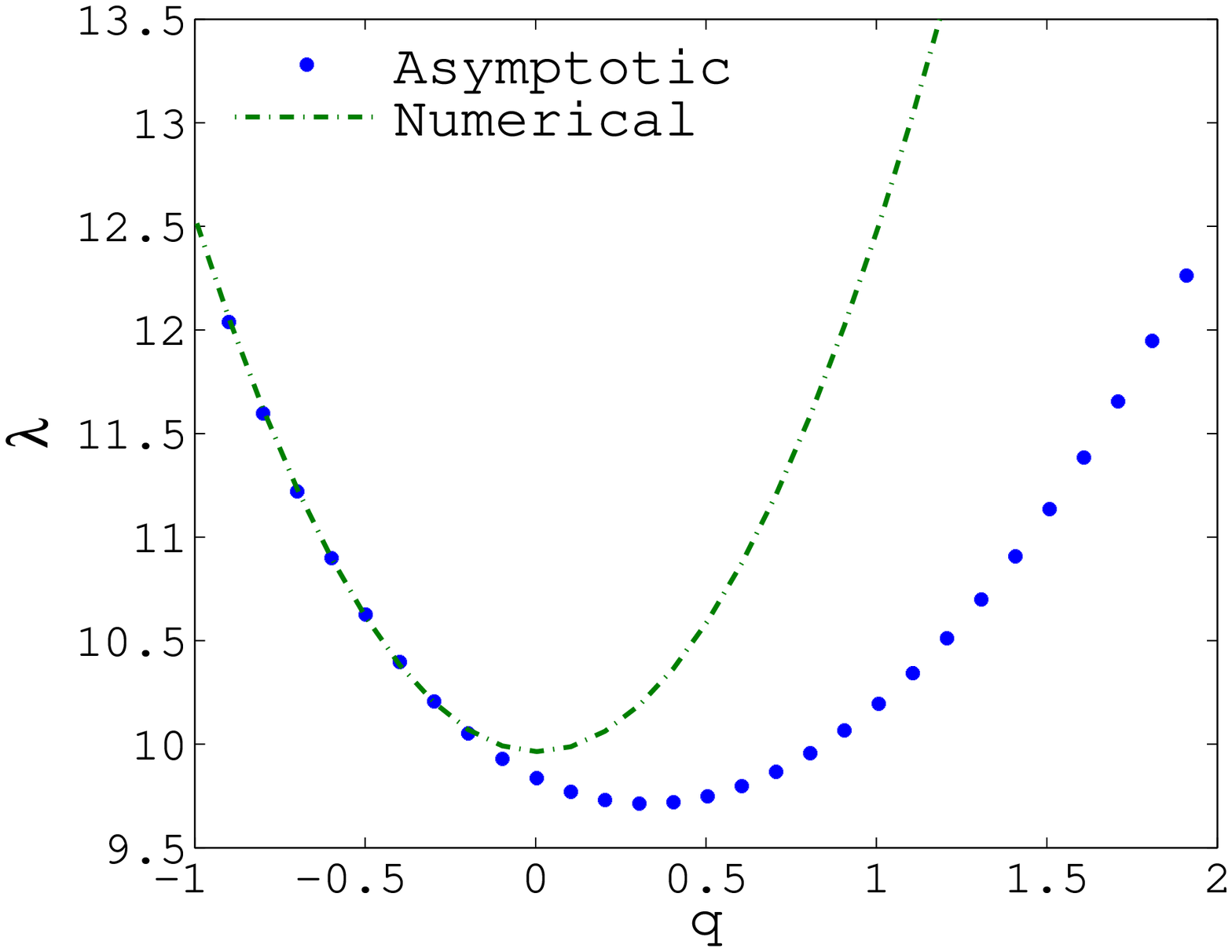}
    \end{center}
\[
\hspace{0.8cm} \textbf{(a)} \hspace{7.5cm} \textbf{(b)} 
\]
    \caption{(a) Equlibria $\bar{\rho}_1^\epsilon$ in three dimensions for $\epsilon = 2,1,0.5,0.2$ ($q = 1,0,-0.5,-0.8$) computed numerically from \eqref{eqn:evp-R1} (see Section \ref{subsect:numerics}). The plain solid line is the limiting profile $\bar{\rho}^{(0)}$ found by inverse iteration from \eqref{eq:nd1}. As $\epsilon \to 0$, $\bar{\rho}_1^\epsilon$ approaches the limiting profile, confirming the asymptotic expansion. (b) Comparison between numerics and asymptotics in three dimensions: eigenvalues obtained numerically from \eqref{eqn:evp-R1} (dots) and from the asymptotic expansion \eqref{eq:lambda1eps} valid at order $O(\epsilon^2)$. The two sets of results agree very well for small $\epsilon$ ($q$ close to $-1$).}
\label{fig:higherdsmall}
\end{figure}
\smallskip

{\bf Remark.} An implication of the above study is that, in all dimensions, the radius $R=\lambda^{-1/(n+q-2)}
=\lambda^{-1/\epsilon}$ vanishes exponentially fast as $q \searrow 2-n$, and as a result, the true steady state 
$\bar{\rho}$ converges to a Dirac delta function. 


\section{Discussion}
\label{sect:discussion}

We have studied the aggregation model \eqref{eq:aggre} with potentials $K$ that contain short-range Newtonian repulsion and long-range power-law attraction. The main merit of the family of potentials considered here is that it leads to solutions which have biologically relevant features, such as finite densities, sharp boundaries and long lifetimes \cite{M&K}. Finding such solutions to model \eqref{eq:aggre} has been indicated as a ``challenge" in previous works \cite{Topaz:Bertozzi,LiRodrigo2009} and the literature addressing this issue has been very scarce.

Well-posedness of solutions to \eqref{eq:aggre} was studied by analogy with incompressible fluid equations \cite{MajdaBertozzi}, using the Lagrangian (particle) formulation of the model.  The main object of the present work, i.e., equilibria supported on a ball, was investigated through a variety of analytical, numerical and asymptotic methods applied to the integral equation  \eqref{eqn:int-eq}. We derived existence and uniqueness of such equilibria using the Krein-Rutman theorem and established qualitative properties such as monotonicity and radial symmetry, by the method of moving planes.  

The numerical results confirm the analytical findings and also suggest that the equilibria studied in this work are global attractors for the dynamics of \eqref{eq:aggre}. We formulate this observation as a conjecture and we plan to address it in future work. A possible approach is to use the fact that the aggregation equation \eqref{eq:aggre} represents a gradient flow with respect to the energy
\begin{equation*}
\label{eq:energy}
E[\rho] = \frac{1}{2}\int_{\mathbb{R}^n}\int_{\mathbb{R}^n} K(x-y)\rho(x)\rho(y)dydx.
\end{equation*}
However the energy is not convex, therefore its  global minimizers cannot be characterized easily. Some recent progress in this direction was done in \cite{BurgerDiFrancescoFranek} in the context of aggregation models with long-range attraction and quadratic diffusion. 

The asymptotic results revealed some very interesting features of solutions to \eqref{eqn:int-eq}. In particular, as the exponent $q$ of the power-law attraction approaches $\infty$, the radii $R$ of the support approach a constant value $0.5$, and the density concentrates on a $\delta$-sphere. Distributions on spheres (uniform, as well as surprisingly complex patterns) have been studied recently \cite{Brecht_etal2011,Balague_etal} using potentials with power-law repulsion and attraction. Concentrations on $\delta$-spheres typically represent equilibrium solutions of the aggregation model. It is not surprising in fact that the non-convex energy $E$ has multiple stationary points, adding to the difficulties in studying its equilibria, as indicated before. Aggregations on spherical shells could be stable or unstable, depending on the exponents of the repulsive and attractive power-laws~\cite{KoSuUmBe2011, Balague_etal}. Choosing the repulsion component in Newtonian form, as in this paper, seems to rule out concentrations on spheres from the possible asymptotic behaviours of solutions to \eqref{eq:aggre}, except  in the limit $q\to \infty$.

The limit $q\to 2-n$ of solutions to \eqref{eqn:int-eq} is interesting for its own sake. Weakly singular integral operators are subjects of many articles and textbooks (see \cite{Vainikko1993} and references therein), but a careful asymptotic study of the eigenvalue problem \eqref{eqn:int-eq}, as the singularity approaches the critical value $2-n$, is missing from the literature. We studied the scaled problem \eqref{eqn:evp-R1} for $R=1$  and showed that eigenvalues $\lambda$ approach a constant, while the corresponding eigenfunctions approach a limiting profile $\bar{\rho}^{(0)}$, as $q \to 2-n$ (see Figures \ref{fig:ldsmall} and \ref{fig:higherdsmall}). Consequently, from the scaling \eqref{eqn:rho-scale}-\eqref{eqn:R-gen}, solutions to \eqref{eqn:int-eq} approach a Dirac $\delta$ in the limit. The findings are consistent with works that consider blow-up in aggregation models with purely attractive potentials \cite{FeRa10,BertozziCarilloLaurent}, in particular recent works that consider Newtonian potentials \cite{BertozziLaurentLeger}.

Finally, we want to comment briefly on the biologically unrealistic feature of the potential \eqref{eq:kernel}, that is, the growth of attraction with distance, when $q>0$. As discussed in more detail in \cite{FeHuKo11}, the dynamics of \eqref{eq:aggre} remains unchanged if the potential is modified in an arbitrary way outside a ball of $\mathbb{R}^n$, with a sufficiently large radius that depends on the initial conditions only. This can be inferred from the property of the density to have uniform (in time) support; this property was shown to hold for $q>1$ and it is believed to hold for all $q>2-n$ in fact.  Provided the radius of the support of the density $\rho(x,t)$ is bounded by $R_{\text{max}}$, where $R_{\text{max}}$ depends only on the initial conditions, but not on time, the potential $K(r)$ can be taken to be zero (or exponentially decaying) for $r>2 R_{\text{max}}$, without changing the dynamics.  We refer the reader to \cite{FeHuKo11} for a numerical illustration of this issue.


\section{Appendix}

\paragraph{Proof of Theorem 2.1}
We show that the operator $\mathcal{F}: \mathcal{O}_{L} \to \mathcal{B}$ is bounded, i.e., $\| \mathcal{F}(X) \|_{1,\gamma} <\infty$, for all $X \in \mathcal{O}_{L}$. We outline the main steps and refer to Chapter 4 \cite{MajdaBertozzi} for details and various technical calculus inequalities.

Write $F(X)$ as 
\[
F(X) = v \circ X.
\]
Using the calculus inequality  (Lemma 4.1 \cite{MajdaBertozzi}) 
\begin{equation}
\label{eqn:calc1}
|XY|_{\gamma} \leq \| X \|_{L^\infty} |Y|_\gamma + |X|_\gamma \| Y \|_{L^\infty},
\end{equation}
we estimate:
\[
\| F(X) \|_{1,\gamma} \leq \| v \|_{L^\infty} + \left( \| \nabla v \|_{L^\infty} + |\nabla v |_{\gamma} \right) \| X \|_{1,\gamma}.
\]
As  $\| X \|_{1,\gamma} \leq L$ for $X \in \mathcal{O}_L$, it remains to bound  $\| v \|_{L^\infty}$, $\| \nabla v \|_{L^\infty}$, and $|\nabla v |_{\gamma} $. 

We first inspect the repulsion component of $v$ (see \eqref{eqn:v2}, \eqref{eqn:f-exp} and \eqref{eqn:kr}) and estimate  $\| k \ast \rho\|_{\infty}$, $\| P[\rho] \|_{L^\infty}$ and $| P[\rho] |_\gamma$, where $P= \nabla k$ and $P[\rho]$ is the principal-value SIO:
\[
P[\rho](x) = \text{PV} \int P(x-y) \rho(y) dy.
\]
The first of these terms can be bounded as follows:
\begin{align}
\label{eqn:estimate1}
|k \ast \rho (x) | & \leq \frac{1}{n \omega_n} \int_{|x-y|<1} \frac{1}{|x-y|^{n-1}} \rho(y) dy + \frac{1}{n \omega_n}\int_{|x-y|>1} \frac{1}{|x-y|^{n-1}} \rho(y) dy \nonumber \\ 
& \leq  \| \rho \|_{L^\infty} + \frac{M}{n \omega_n}.
\end{align}
To bound $P[\rho]$ we follow the proof of Lemma 4.6 \cite{MajdaBertozzi}. We split the integral:
\[
|P[\rho](x)| = \underbrace{\text{PV} \int_{|x-y|<\epsilon} P(x-y) \rho(y) dy}_{I_1(x)} + \underbrace{\int_{|x-y|>\epsilon} P(x-y) \rho(y) dy}_{I_2(x)}.
\]
The kernel $P$ has mean-value zero on the unit sphere,
\[
\int_{|x|=1} P ds =0,
\]
which enables us to rewrite $I_1$ as 
\[
I_1(x) = \text{PV} \int_{|y|<\epsilon} P(y) (\rho(x-y)-\rho(x)) dy.
\]
Hence
\begin{align*}
|I_1(x)| &\leq \int_{|y|<\epsilon} P(y) \frac{|\rho(x-y)-\rho(x)|}{|y|^\gamma} |y|^\gamma dy \\
& \leq c |\rho|_\gamma \int_{|y|<\epsilon} |y|^{-n+\gamma} dy,
\end{align*}
and 
\[
|I_1(x)| \leq c  |\rho|_\gamma \epsilon^\gamma.
\]
Here and below, $c$ denotes a generic constant. For the second integral, we estimate
\begin{align*}
|I_2(x)| &\leq \int_{\epsilon<|x-y|<1}  | P(x-y) | \rho(y) dy + \int_{|x-y|>1}  | P(x-y) | \rho(y) dy \\
& \leq c \| \rho \|_{L^\infty} \log\left( \frac{1}{\epsilon}\right) + M.
\end{align*}
Conclude
\begin{equation}
\label{eqn:estimate2}
\|P[\rho]\|_{L^\infty}  \leq c \left (  |\rho|_\gamma \epsilon^\gamma + \| \rho \|_{L^\infty} \log (1/\epsilon) \right) +M.
\end{equation}
Finally, using the same argument as in Lemma 4.6 \cite{MajdaBertozzi}, one can show
\begin{equation}
\label{eqn:estimate3}
|P[\rho]|_{\gamma} \leq c |\rho|_{\gamma}.
\end{equation}
Estimate the seminorm  $|\rho|_{\gamma}$ as in the proof of Proposition 4.2 \cite{MajdaBertozzi}, by using 
\eqref{eqn:rhoJ} and calculus inequality \eqref{eqn:calc1}:
\begin{equation}
\label{est:rho-gamma}
| \rho |_{\gamma} \leq \| \rho_0 \circ X^{-1}\|_{L^\infty} |\text{det } \nabla_x X^{-1} |_{\gamma} + | \rho_0 \circ X^{-1}|_{\gamma} \| \text{det } \nabla_x X^{-1} \|_{L^\infty}.
\end{equation}
Using the calculus inequalities 
\begin{equation}
\label{eqn:lemma43}
|f \circ Y |_\gamma \leq |f|_\gamma \, \| \nabla Y \|_{L^\infty}^\gamma, 
 \end{equation}
\begin{equation}
\label{eqn:lemma42}
\| Y^{-1}\|_{1,\gamma} \leq c \|Y\|_{1,\gamma}^{2n-1}, 
\end{equation}
from Lemmas 4.3 and 4.2 \cite{MajdaBertozzi}, respectively, estimate
\begin{align*}
| \rho_0 \circ X^{-1}|_{\gamma} & \leq |\rho_0|_{\gamma} \,  \| \nabla_x X^{-1} \|_{L^\infty}^{\gamma} \\
& \leq c  |\rho_0|_{\gamma} \|X\|_{1,\gamma}^{\gamma(2n-1)},
\end{align*}
and
\[
|\text{det } \nabla_x X^{-1} |_{\gamma} \leq c  \|X\|_{1,\gamma}^{2n-1}.
\]
As $X \in \mathcal{O}_L$, $\| \text{det } \nabla_x X^{-1} \|_{L^\infty}<L$ and $\|X\|_{1,\gamma} <L$. Return to \eqref{est:rho-gamma} to find:
\[
| \rho |_{\gamma} \leq C(L) \left(  \| \rho_0 \|_{L^\infty} + |\rho_0|_{\gamma}  \right).
\]
Using the above bound on $| \rho |_{\gamma}$ and the uniform bound on $\| \rho\|_{L^\infty}$ derived in Section \ref{subsect:bounds}, we infer from \eqref{eqn:estimate1}-\eqref{eqn:estimate3} that the repulsion component of $\mathcal{F}$ yields a bounded operator. The singularity of the attraction component is milder than that of the repulsion and does not break the existing estimates for the singular repulsion kernel $k$. We conclude that $\mathcal{F}$ is bounded.

To prove that $\mathcal{F}$ is Lipschitz continuous, we show that $\mathcal{F}'(X)$ is bounded as a linear operator from $\mathcal{O}_L$ to $\mathcal{B}$, i.e., $\| \mathcal{F}'(X) \| < \infty $, for all $X \in \mathcal{O}_L$. Calculate $\mathcal{F}'(X)$ using \eqref{eqn:F-qgen} and \eqref{eqn:f-exp}:
\begin{align}
\label{est:Fprime}
\mathcal{F}'(X)Y &=  \left. \frac{d}{d \epsilon}  \mathcal{F}(X + \epsilon Y)  \right \vert_{\epsilon =0} \nonumber \\
& = \int \nabla f (X(\alpha) - X(\beta)) (Y(\alpha) - Y(\beta)) \rho_0(\beta) d \beta,
\end{align}
where we suppressed the time dependence for convenience.

To estimate the component of $\mathcal{F}'(X)Y$ due to repulsion one could follow the proof of Lemma 4.10 in \cite{MajdaBertozzi}. The attraction component is milder and does not break the estimates. It can be shown that 
\[
\| \mathcal{F}'(X)Y \|_{1,\gamma} \leq C(L) \left(  \| \rho_0 \|_{L^\infty} + |\rho_0|_{\gamma}  \right) \| Y \|_{1,\gamma},
\]
which proves the boundedness in the operator norm of $\mathcal{F}'(X)$. 

A key observation is that the term $Y(\alpha)-Y(\beta)$ in \eqref{est:Fprime} compensates for the singularity of $\nabla f$. We present here the estimate of $\| \mathcal{F}'(X)Y \|_{L^\infty}$, which requires in fact some changes to the proof of Lemma 4.10 \cite{MajdaBertozzi}. We do not present the estimates of $\| \nabla_{\alpha }\mathcal{F}'(X)Y \|_{L^\infty}$ and $| \nabla_{\alpha }\mathcal{F}'(X)Y |_{\gamma}$, we refer instead to the calculations in  \cite{MajdaBertozzi}.

Use a change of variable, $x=X(\alpha)$, $y=X(\beta)$, and split the  repulsion component of $ \mathcal{F}'(X)Y$ from \eqref{est:Fprime} 
\[
\int \nabla k (x-y) (Y(X^{-1}(x))-Y(X^{-1}(y))) \rho(y) dy =\underbrace{\int_{|x-y|<1}}_{J_1} + \underbrace{\int_{|x-y|>1}}_{J_2}.
\]
The procedure is similar to what we did to derive \eqref{eqn:estimate1}. From mean-value theorem, we have
\[
|Y(X^{-1}(x))-Y(X^{-1}(y))| \leq \| (\nabla_\alpha Y \circ X^{-1}) \nabla_x X^{-1}\|_{L^\infty} |x-y|. 
\]
As $|\nabla k(x)|  \leq c |x|^{-n}$,
\begin{align*}
|J_1| &\leq  c \| (\nabla_\alpha Y \circ X^{-1}) \nabla_x X^{-1}\|_{L^\infty}  \| \rho \|_{L^\infty} \int_{|x-y|<1} \frac{1}{|x-y|^{n-1}} dy \\
& \leq C(L) \| \nabla_\alpha Y \|_{L^\infty},
\end{align*}
where we also used \eqref{eqn:lemma42}, $X \in \mathcal{O}_L$, and the uniform bound on $\| \rho \|_{L^\infty}$ in the second inequality. 

The outer integral satisfies
\begin{align*}
|J_2| &\leq c \| (\nabla_\alpha Y \circ X^{-1}) \nabla_x X^{-1}\|_{L^\infty} \int_{|x-y|>1} \rho(y) dy \\
& \leq C(L) M  \| \nabla_\alpha Y \|_{L^\infty},
\end{align*}
where we used conservation of mass.

Combine the two estimates for $J_1$ and $J_2$ and argue that the attraction component would not break these estimates, to find
\[
\| \mathcal{F}'(X)Y \|_{L^\infty} \leq C(L) \| Y \|_{1,\gamma}.
\]

\bibliographystyle{plain}
\bibliography{lit}
\end{document}